\documentclass[oneside,11pt]{article}
\usepackage{epsfig, amsfonts, amsmath, amsthm,amsbsy}
\usepackage{color}
\usepackage{subfigure}
\newtheorem{theorem}{Theorem}

\newtheorem{corollary}{Corollary}

\def\ve{\varepsilon}
\def\vr{\varepsilon}

  \DeclareMathOperator\erfc{erfc}

\begin{document}

\title{Parameter-uniform numerical methods for singularly perturbed  parabolic problems with incompatible boundary-initial data}

\author{J.L. Gracia\footnote{IUMA - Department of Applied Mathematics, University of Zaragoza, Spain. email: jlgracia@unizar.es. The
research of this author was partly supported by the Institute of Mathematics and Applications (IUMA), the
project MTM2016-75139-R and the Diputaci\'on General de Arag\'on (E24-17R).} \, and \,
E. O'Riordan\footnote{School of Mathematical Sciences, Dublin City
University, Ireland. \qquad
email: eugene.oriordan@dcu.ie}}

\date{\today}

\maketitle

\begin{abstract}

Numerical approximations to the solution of a linear singularly perturbed parabolic reaction-diffusion problem  with incompatible bound\-ary-initial data are generated, The method involves combining the computational solution of a  classical finite difference operator on a tensor product of two piecewise-uniform Shishkin meshes  with an analytical function that captures the local nature of the incompatibility. A proof is given to show almost first order parameter-uniform convergence of these numerical/analytical approximations. Numerical results are given to illustrate the theoretical error bounds.

\end{abstract}

\maketitle

\section{Introduction}

We examine singularly perturbed parabolic problems in one space dimension, with an incompatibility between the initial condition and a boundary condition. These problems arise in  mathematical models in fluid dynamics \cite{han} and, in particular, models for flow in porous media \cite{han2}. The solutions of these problems typically exhibit boundary layers, initial layers and initial-boundary layers. In this paper we are interested in constructing a parameter-uniform numerical algorithm \cite{fhmos} for this class of singularly perturbed problems.

Numerical methods  generate finite dimensional  approximations~$U^N$ (where $N$ is the number of mesh elements used in each coordinate direction) to the continuous solution $u$ at the selected nodal points within the continuous domain $\bar Q$. A global approximation $\bar U^N $ can also be created, using a user chosen choice of interpolating basis functions. In this paper, we shall simply employ bilinear basis functions. Parameter-uniform numerical methods \cite{fhmos}  satisfy  a theoretical error bound of the form:
\[
\Vert u - \bar U^N \Vert _{\bar Q} \leq CN^{-p},\quad p >0;
\]
where $\Vert \cdot \Vert _{\bar Q} $ is the $L_\infty$ norm  on the closed domain $\bar Q$, $C$ is a generic constant,  which depends on the problem data but  is independent of $N$ and the singular perturbation $\ve$. We emphasize that this error bound estimates the pointwise  error at all points in the domain $\bar Q$ of the continuous solution. Parameter-uniform convergence at the nodes is a necessary, but not a sufficient, condition for parameter-uniform global convergence. If a numerical method is parameter-uniform at the nodes, then the distribution of the mesh points and the selected form of interpolation will determine whether the method is globally parameter-uniform or not.

Within the literature on singularly perturbed problems, there are two common approaches to designing a parameter-uniform method: fitted operator (see e.g. \cite{mos2}) or fitted mesh methods \cite{fhmos}. Fitted operator methods tend to use a quasi-uniform discretization of the domain and incorporate analytical information about the solution character within the  layers, into the choice of special basis functions (in a finite element framework) or (in the case of finite differences) by choosing a special finite difference operator that is exact in the case of constant coefficient one dimensional model problems. On the other hand, fitted mesh methods use {\it a priori} information about the layer structure to construct an appropriate non-uniform distribution of the mesh points.

For some classes of singularly perturbed problems with boundary layers, fitted operator methods on a uniform mesh exist which satisfy a  parameter-uniform error bound at the nodes, but these fitted operator methods are not globally parameter-uniformly convergent \cite{fhmos}, when some form of polynomial interpolation is employed. In the case of  one-dimensional problems not containing characteristic  boundary layers, global convergence can be guaranteed if one  subsequently incorporates  exponential splines to form the interpolant~\cite{styor}; but this form of non-polynomial spline interpolation is difficult to extend to elliptic problems in higher dimensions. Moreover, a nodally parameter-uniform fitted operator method cannot be constructed for a class of singularly perturbed heat equations, if one only uses a uniform mesh \cite{fhmos,cwi987b}. This same impasse is faced when dealing with elliptic problems, whose solutions contain characteristic boundary layers.
However, a fitted piecewise-uniform Shishkin mesh coupled with a classical discrete operator produces a parameter-uniform numerical method for  singularly perturbed heat equations \cite{ria} and for elliptic problems with characteristic layers~\cite{hemkerb}. Moreover, parameter-uniform numerical methods, using an appropriate Shishkin mesh  have been designed for a wide class of singularly perturbed problems \cite{review-GIS}. These problem classes include  problems with both boundary and initial layers.

To establish pointwise parameter uniform error bounds on numerical approximations to the solutions of singularly perturbed parabolic problems, most publications assume second level compatibility conditions and sufficient regularity of the data so that the solution is in $C^{4+\gamma}(\bar Q)$
\footnote{The space $C^{n+\gamma}(\bar Q ) $ is
the set of all functions, whose derivatives of
order $n$ are H\"{o}lder continuous of degree $\gamma >0$. That is,
\[
C^{n+\gamma } ( \bar Q ) := \left \{ z : \frac{\partial ^{i+j} z}{
\partial x^i
\partial t^j } \in C^{\gamma }(\bar Q), \ 0 \leq i+2j \leq n \right \} .
\]}
 in the closed domain $\bar Q$.  Interested readers are referred, for example, to \cite{Zhemukhov2}. In the case of singularly perturbed parabolic problems in one space dimension and using appropriate fitted meshes, these compatibility constraints can be relaxed to zero order, without an adverse effect on the rate of uniform convergence \cite{Zhemukhov1}. Hence, parameter-uniform numerical methods exist when the boundary and initial data are simply assumed to be continuous.

However, there are difficulties with constructing a fitted mesh method for problems with an incompatibility between the initial  and a boundary condition; or for a problem with a discontinuity in a boundary or the initial condition \cite{hemker, hemker2}. Hemker and Shishkin \cite{hemker2} constructed a fitted operator method on a uniform mesh, which is nodally parameter-uniform for a singularly perturbed heat equation with a discontinuity in the initial condition; but the method is not globally parameter-uniform, using bilinear interpolation. An extension of this fitted operator method to a fitted operator method on a fitted piecewise-uniform mesh was constructed in \cite{bulg2012}, but this again failed to be parameter-uniform globally, using bilinear interpolation. The interpolation  failed to produce an accurate global approximation in a neighbourhood of the point, where  the initial condition and a boundary condition were incompatible.

Another approach to dealing with a problem having discontinuous data is to replace the problem with a regularized problem with continuous data \cite{amc}. This approach is strongly related to the penalty method discussed in \cite{temam}. For example, the problem
\begin{eqnarray*}
u_t-\ve u_{xx} =f(x,t),\qquad x,t >0;\\ u(0,t) = 0,\ t >0;\quad u(x,0) = 1,\ x >0
\end{eqnarray*}
can be approximated by the solution $u^{r}$ of the regularized problem:
\begin{eqnarray*}
u^{r}_t-\ve u^{r}_{xx} =f(x,t),\qquad  x,t >0; \\ u^{r}(0,t) = 0,\ t >0; \quad u^{r}(x,0) = 1 - e^{-x/\sqrt{\ve}},\ x \geq 0.
\end{eqnarray*}
Parameter-uniform numerical approximations to $u^r$ can be generated, but (see \cite{amc}) these approximations are only accurate approximations to $u$ outside a neighbourhood of the point $(0,0)$.  In other words, this approach will not generate parameter-uniform global approximations to the original problem with an incompatibility between the boundary and initial data.

 In this paper, we examine an alternative approach to dealing with this problem class, which uses an idea examined numerically in~\cite{Flyer} in the non-singularly perturbed case
(set $\ve =1$). Given a differential operator $L$, the solution $u$  of the continuous problem
\[
L u =f, \ \hbox{in}\ Q , \quad u =g , \ \hbox{on} \ \bar Q \setminus Q =: \partial Q, \quad \hbox{where} \quad g \notin
C^0(\partial Q);
\]
is written as the sum of two components $u=s+y $. The function $s$  matches the incompatibility in the solution $u$ and the other term $y$ satisfies
the singularly perturbed problem
\[
L y =f-L s, \ \hbox{in} \ Q ,\quad y =g -s, \ \hbox{on}\  \bar Q \setminus Q , \quad \hbox{where} \quad g -s\in C^0(\partial Q).
\]
 In this paper, we design a parameter-uniform numerical method for this secondary problem, which generates a global approximation $\bar Y$ to $y$. In this way, we can generate parameter-uniform numerical approximations $s +\bar Y$ to the solution $u$ of a singularly perturbed problem with an incompatibility between the initial condition and a boundary condition.
Note that here we restrict the discussion to problems in one space dimension. Extensions of the method to two space dimensions are not obvious \cite{temam} and require further investigation.

The rest of the paper is structured as follows. In the next section, we define the problem class to be examined, we decompose the continuous solution into various components and we derive parameter-explicit bounds on the derivatives of each of these components. In  Section 3, we construct the numerical method and we establish a parameter-uniform bound on the error. In  Section 4, we present the results of some numerical experiments with a representative test problem. For the sake of completeness, we write out the compatibility conditions of levels zero, one and two in the first appendix. Finally, in a second appendix, we present some properties of an analytical function which are used in the proof of Theorem 3.

\noindent {\bf Notation.}
%The space ${\mathcal C}^{0+\gamma}(D )$, where $D \subset \mathbf{R}^2$ is an open set, is the set of all functions that are H\"{o}lder continuous of degree $\gamma $ with respect to the metric $\Vert \cdot \Vert, $
%where for all ${\bf p}_1=(x_1,t_1), \ {\bf p}_2=(x_2,t_2) \in \mathbf{R}^2$, $
%\Vert {\bf p}_1- {\bf p}_2 \Vert^2  = (x_1-x_2)^2 + \vert t_1 -t_2 \vert$.
%For $f$ to be in ${\mathcal C}^{0+\gamma}(D ) $ the following semi-norm needs to be finite
%\[
%\lceil f \rceil _{0+\gamma , D} = \sup _{{\bf p}_1
 %\neq {\bf p}_2, \ {\bf p}_1, {\bf p}_2\in D}
%\frac{\vert f({\bf p}_1) - f({\bf p}_2) \vert}{\Vert {\bf p}_1- {\bf p}_2 \Vert ^\gamma} .
%\] The space ${\mathcal C}^{n+ \gamma}(D ) $ is defined by
%\[
%{\mathcal C}^{n+\gamma }( D ) = \{ z : \frac{\partial ^{i+j} z}{
%\partial x^i
%\partial t^j } \in {\mathcal C}^{0+\gamma }(D), \ 0 \leq i+2j \leq n \},
%\]
%and $\Vert \cdot \Vert _{n + \gamma}, \ \lceil \cdot \rceil _{n+\gamma}$ are the associated norms and semi-norms.
Throughout the paper,   $C$  denotes a generic constant that is independent of the singular perturbation parameter $\ve$ and of all discretization parameters.
  The $L_\infty$ norm on the domain $D$ shall  be denoted by $\Vert \cdot \Vert _D$  and the subscript is omitted if the domain is $\bar Q$.

\section{Continuous problem}

Consider  the singularly perturbed parabolic problem:
Find
$u: \bar Q \rightarrow \mathbb{R}$ with $Q:=(0,1)\times(0,T],$ such that
\begin{subequations} \label{Cproblem}
\begin{align}
&Lu := \vr (u_t-u_{xx})+b(x,t)u=f(x,t), \quad (x,t) \in  Q; \label{Lu}\\
&u(0,t)=g_L(t), \ u(1,t)=g_R(t) \ t\ge 0, \quad u(x,0)=\phi(x), \ 0<x<1; \label{bound+init-cond}\\
&\phi(0^+)\ne g_L(0), \quad \phi(1^-)=g_R(0),\label{discont}\\ &b_x(0,0) =0 \quad \hbox{and} \quad
b(x,t)> \beta > 0, \ \forall t \geq 0; \label{b_x(0,0)=0} \\
&f,b \in C^{4+\gamma }( \bar Q ),  \ g_L, g_R \in C^2[0,T], \ \phi \in C^4[0,1]. \label{Reg-data}
\end{align}
\end{subequations}
Observe that the solution of this problem is discontinuous at the corner $(0,0)$ of the domain $\bar Q$.
We define the related constant coefficient differential operator
\begin{equation}\label{L0}
L_0z:= \vr (z_t-z_{xx})+b(0,0)z,
\end{equation}
so that (by (\ref{b_x(0,0)=0}))  \[
\vert (L-L_0) z(x,t) \vert \leq C(x^2+t)\vert  z (x,t)\vert.
\] It is important to point out that the coefficient $b(x,t)$ can depend on both the space and time  variables. In the special case where this coefficient only depends on time, then the singularity associated with the incompatibility at $(0,0)$ can be found analytically. 

 We also assume the compatibility conditions at the point $(1,0)$
\begin{align}
&  \vr (g_R'(0^+)-\phi_{xx}(1^-))+b(1,0)g_R(0)=f(1,0); \label{First-R}\\
& \vr (g_R''(0^+) -\phi _{xxxx}(1^-))+b(1,0)(g'_R(0)+\phi _{xx}(1^-)) \nonumber \\
&  \hspace{0.2cm}+ b_t(1,0)g_R(0)+2b_{x}(1,0)\phi _x(1^-) +b_{xx}(1,0)\phi(1^-) =\bigl( f_t+f_{xx} \bigr)(1,0); \label{Second-R}
\end{align}

Here we simply assume these additional compatibility conditions in order to concentrate on the issues near $(0,0)$, associated with the lack of corresponding compatibility conditions being assumed at $(0,0)$. The numerical method presented below will satisfy the same error bound, established in Theorem 4, even when the data does not satisfy the constraints (\ref{First-R}), (\ref{Second-R}). In this section, the solution $u$ is decomposed in a sum of terms, some associated with the layers in the solution and some terms (denoted below by $A_0z_0(x,t)+A_1z_1(x,t)+A_2z_2(x,t)$) associated with the lack of compatibility being assumed at $(0,0)$. If we did not assume (\ref{First-R}), (\ref{Second-R}), then additional terms of the form  $A^R_1z_1(1-x,t)+A^R_2z_2(1-x,t)$ would be included in the expansion of the continuous solution; and the influence of these additional terms on the numerical analysis, could be tracked in the exact same way as the terms $A_1z_1(x,t)+A_2z_2(x,t)$ are handled in the error analysis below. Hence, it is solely for the sake of clarity of exposition in this section of the paper, that we assume the compatibility conditions
(\ref{First-R}), (\ref{Second-R}).

Decompose the solution of~\eqref{Cproblem} into the sum
\begin{equation}\label{A0}
u= A_0e^{-\frac{b(0,0) t}{\vr}}\erfc\left(\frac{x}{2\sqrt{t}} \right)+y, \quad A_0:= g_L(0) - \phi (0^+),
\end{equation}
where $\erfc(z)$ is the complimentary error function
$$
\erfc(z):= \frac{2}{\sqrt{\pi}} \int _{s=z}^\infty e^{-s^2} \ ds.
$$
Note that the function \[
z_0(x,t):= e^{-\frac{b(0,0) t}{\vr}}\erfc\left(\frac{x}{2\sqrt{t}} \right)
\] is the first of a family of functions defined as the solutions of the constant coefficient homogeneous
quarter plane problems, where for all $n =0,1,2, \ldots$
\begin{equation}\label{zn}
  L_0z_n =0, \ x,t >0, \
z_n(0,t)=t^ne^{-\frac{b(0,0) t}{\vr}}, \  t>0, \ z_n(x,0)=0, \ x>0.
\end{equation}
In Appendix 2, we explicitly write out several derivatives of these functions and we discuss the regularity of these functions.

In this section, we  establish {\it a priori} bounds on the derivatives of the continuous function  $y:= u -A_0z_0$, which satisfies the problem
\begin{subequations}\label{eq:ComponentY2}
\begin{eqnarray}
&Ly=f(x,t)-A_0(L-L_0)z_0(x,t), \ \hbox{ in } Q; \\
&y(0,t)=g_L(t) -A_0e^{-\frac{b(0,0) t}{\vr}}, \quad y(1,t)=g_R(t)-A_0 z_0(1,t), \quad  t\geq 0; \\
&y(x,0)=\phi(x), \quad 0 < x < 1.
\end{eqnarray}
\end{subequations}
We  introduce extended domains, where various subcomponents of the solution $y$  will be defined:
For arbitrary positive constants $p,q,r $,
\begin{eqnarray*}
\bar Q^* := [-p,1+q] \times [-r, T]; \ \bar Q^*_S := [0,1] \times [-r, T]; \
%\bar Q^*_R := [-p,1] \times [-r, T]; \qquad
 \bar Q^*_B := [-p,1+q] \times [0, T].
\end{eqnarray*}
To avoid excessive notation, we shall denote smooth extensions of the functions $b,f$ to some larger domain  by $b^*,f^*$ (such that $f^* \vert _{\bar Q} \equiv f$), even those these extensions will be taken over different domains.

The solution of (\ref{eq:ComponentY2}) can be decomposed into a sum of a regular component $v$ and several layer components $w$ (with a subscript to identify the location of the layer)  defined as follows:
\begin{subequations}\label{decomp}
\begin{equation}
y= v+w_L+w_R+ w_I+w_{IB};
\end{equation}
where the regular component $v$ satisfies the problem
\begin{equation}
L^*v^*=f^*, \quad \hbox{ in } Q^*,\quad v^*=v^*, \ \hbox{ on }  \partial  Q^*. \end{equation}
The boundary/initial values for the regular component are determined from the reduced solution $v_0$ and a correction $v_1$. We write $v^* =v^*_0+\vr v^*_1$, where the reduced solution $v_0$ and the correction $v_1$  are defined via
\begin{eqnarray}
&v^*_0=\bigl( \frac{f}{b}\bigr)^* , \hbox{ in } Q^*; \\
&L^*v^*_1 = (v_0)^*_{xx} - (v_0)^*_t,\quad \hbox{ in } Q^*, \quad v^*_1=0, \ \hbox{ on } \partial  Q^* \label{v_1}.
\end{eqnarray}
The boundary layer components $w_L,w_R$ satisfy the homogeneous problems
\begin{eqnarray}
&L^*w^*_R=L^*w^*_L=0, \ \hbox{ in } Q^*_S;\\
& w^*_L(0,t)=(y-v^*)^*(0,t), \  w^*_L(x,-r)=0, \ w^*_L(1,t)=0, \  \hbox{ on } \partial  Q_S^*;  \quad \\
&w^*_R(0,t)=0, \    w^*_R(x,-r)=0, \ w^*_R(1,t)=(y-v^*)^*(1,t), \  \hbox{ on } \partial  Q_S^*. \quad
\end{eqnarray}
The initial layer function $w_I$ satisfies the problem
\begin{eqnarray} \label{initial-layer}
&L^*w^*_I=0  \,  \quad \hbox{in}\quad  Q^*_B; \\
& w^*_I(-p,t)=0, \  w^*_I(x,0)=(y-v^*)^*(x,0), \ w^*_I(1+q,t)=0,  \hbox{ on } \partial  Q_B^*. \label{initial-layer-icbc}
\end{eqnarray}
%\end{subequations}
Having defined the problems over the extended domains, to avoid compatibility issues, all of these components are in $C^{4+\gamma}(\bar Q)$.

Finally,  the initial-boundary  layer component  $w_{IB}$ satisfies the problem
%\begin{subequations}
\begin{eqnarray}
&Lw_{IB}=-A_0(L-L_0)z_0(x,t), \quad (x,t)\in Q; \label{initial-boundary-layera}\\
&w_{IB}(0,t)=-w^*_I(0,t),    \quad w_{IB}(1,t)=-w^*_I(1,t), \quad t \ge 0; \label{initial-boundary-layerb}\\
& w_{IB}(x,0)=-w^*_L(x,0)-w^*_R(x,0), \quad 0 < x < 1. \label{initial-boundary-layerc}
\end{eqnarray}
\end{subequations}
 The regularity of this key  component  $w_{IB}$ is discussed below in Theorem \ref{thm-3}.

\begin{theorem} \label{th:vwLRbounds}
 For all $0 \leq i+2j \leq 4$, we have the following bounds.
\begin{subequations}
For the regular component $v \in C^{4+\gamma}(\bar Q)$,
\begin{eqnarray}\label{vbounds}
\left \Vert \frac{\partial ^{i+j} v}{\partial x^{i}\partial t ^{j}}  \right\Vert  \leq C \left(1+\vr ^{1-(i/2 +j)}\right);
\end{eqnarray}
and, for  all points $(x,t)\in Q$, the boundary layer components $w_L,w_R \in C^{4+\gamma}(\bar Q)$ satisfy
\begin{eqnarray}
&\vert w_L(x,t) \vert \leq C e^{-\sqrt{\frac{\beta}{\vr}}x};  \quad
\vert w_R(x,t) \vert \leq C e^{-\sqrt{\frac{\beta}{\vr}}(1-x)}; \, \label{layer-bounds}\\
& \left  \vert \frac{\partial ^{i+j}}{\partial x^{i}\partial t ^{j}} w_{L}(x,t)\right \vert  \leq C \vr ^{-(i/2 +j)}e^{-\sqrt{\frac{\beta}{\vr}}x},\\
& \left \vert \frac{\partial ^{i+j}}{\partial x^{i}\partial t ^{j}} w_{ R}(x,t) \right \vert \leq C \vr ^{-(i/2 +j)}e^{-\sqrt{\frac{\beta}{\vr}}(1-x)} .\label{explayer-bounds}
\end{eqnarray}
In addition, for $1\leq j \leq 2$, the time derivatives satisfy
\begin{equation}\label{bnds-time.dervs}
\max \left\{ \left\Vert \frac{\partial ^{j}w_{L}}{\partial t ^{j}} \right\Vert, \left\Vert \frac{\partial ^{j}w_{ R}}{\partial t ^{j}}  \right\Vert \right\}\leq C \vr ^{1-j}.
\end{equation}
\end{subequations}
\end{theorem}
\begin{proof}
Using the stretched variables
 \[
\zeta = \frac{x}{\sqrt{\ve}} \quad \hbox{and} \quad \eta = \frac{t}{\ve}
\]
 problem (\ref{v_1}) transforms into the problem
\[
\tilde v _{\zeta \zeta} - \tilde v_\tau +\tilde b \tilde v = \tilde f_1,\quad  \hbox{in} \left(\frac{-p}{\sqrt{\ve}},\frac{1+q}{\sqrt{\ve}}\right)\times \left(\frac{-r}{\ve},\frac{T}{\ve}\right],
\]
where \[
 \tilde v(\zeta, \tau) = v^*_1(x,t) \quad \hbox{and} \quad 
\tilde f_1(\zeta, \tau) = (v_0)^*_{xx} - (v_0)^*_t .
\]
Applying the {\it a priori} bounds \cite{ladyz} on the derivatives of the solution $\tilde v$, and transforming back to the original variables $(x,t)$, we get that
\begin{eqnarray}\label{crude}
\left\Vert \frac{\partial ^{i+j}v_1^*}{\partial x^{i}\partial t ^{j}}  \right\Vert  \leq C \left(1+\vr ^{-(i/2 +j)}\right),\ 0 \leq i+2j \leq 4,
\end{eqnarray}
and we have deduced the bounds (\ref{vbounds}) on the derivatives of the regular component.
A maximum principle, the assumption $b(x,t) > \beta $, the corresponding bounds (\ref{crude}) and the argument from
\cite[Theorem 4]{ria} yield the bounds on the boundary layer components $w_L,w_R$ (\ref{layer-bounds})-(\ref{explayer-bounds}).

Note that
\begin{eqnarray*}
&L^*\frac{\partial w^*_L}{\partial t} = -\frac{\partial b}{\partial t} w^*_L, \ \hbox{ in } Q^*_S;\\
&\frac{\partial w^*_L}{\partial t} (0,t)= \frac{\partial (y-v^*)^*}{\partial t}(0,t), \  \frac{\partial w^*_L}{\partial t}(x,-r)=0, \ \frac{\partial w^*_L}{\partial t}(1,t)=0, \  \hbox{ on } \partial  Q_S^*.
\end{eqnarray*}
Using the stretched variables and the earlier argument, we deduce that
\[
\left\Vert \frac{\partial ^{i+j}}{\partial x^{i}\partial t ^{j}} \Bigl( \frac{\partial w^*_L}{\partial t} \Bigr) \right\Vert  \leq C \left(1+\vr ^{-(i/2 +j)}\right),\ 0 \leq i+2j \leq 2,
\]
which yields the final bounds (\ref{bnds-time.dervs}) on the time derivatives of the layer components.
\end{proof}

\begin{theorem}
\begin{subequations} \label{sub:initial-layer-bounds}
For the initial layer component $w_I\in  C^{4+\gamma}(\bar Q)$ and
\begin{eqnarray}
\vert w_I(x,t) \vert &\leq& C e^{-\frac{ \beta t}{\vr}}, \, \ (x,t)\in Q;\label{initial-layer-bounds}\\
\left\vert \frac{\partial ^{i+j}w_I }{\partial x^{i}\partial t ^{j}}(x,t)  \right\vert  &\leq& C \vr ^{-i/2}\vr ^{-j}e^{-\frac{ \beta  t}{\vr}},\ 0 \leq i+2j \leq 4; \label{initial-layer-bounds-der}\\
\left\vert \frac{\partial ^{i} w_I}{\partial x ^{i}} (x,t) \right\vert  &\leq& C (1+\vr ^{1-i/2}) e^{-\frac{ \beta t}{\vr}},\quad i=1,2,3,4.
\label{initial-layer-bounds-der-improved}
\end{eqnarray}
\end{subequations}
\end{theorem}

\begin{proof}
Note that for the initial condition $ \Vert w^*_I(x,0) \Vert \leq C$,  the boundary conditions $w^*_I(-p,t)=w^*_I(1+q,t)=0$ and
\[
L^*e^{-\frac{ \beta t}{\vr}} =(b^*(x,t) -\beta) e^{-\frac{ \beta t}{\vr}} \geq 0.
\]
The bound (\ref{initial-layer-bounds}) follows from the maximum principle.

To deduce bounds on the derivatives of $w_I$, we repeat the argument from \cite{ria}. Transforming to  the stretched variables $\zeta = x/\sqrt{\ve}, \eta = t/\ve $ we have
\begin{eqnarray*}
\left( \frac{\partial }{\partial \eta} - \frac{\partial ^2}{\partial \zeta  ^2} + \tilde b I\right)(\tilde w^*_I)=0, \quad  \hbox{in } \left(-\frac{p}{\sqrt{\ve}},\frac{1+q}{\sqrt{\ve}} \right)\times \left(0,\frac{T}{\ve} \right], \\
\left \vert \frac{\partial ^i \tilde w^*_I}{\partial \zeta  ^i} (\zeta,0) \right \vert \leq C, \qquad 0 \leq i \leq 4.
\end{eqnarray*}
From the interior Schauder estimates \cite[\S4.10]{ladyz}, we have that for any neighbourhood $\tilde N_\delta \in \left(-\frac{p}{\sqrt{\ve}},\frac{1+q}{\sqrt{\ve}}\right) \times \left(1,T/\ve \right)$
\[
\left \Vert \frac{\partial ^{i+j} \tilde w^*_I}{\partial \eta ^{j}  \partial \zeta  ^i}  \right \Vert _{\tilde N_\delta} \leq C\Bigl \Vert \tilde w^*_I\Bigr\Vert _{\tilde N_{2\delta}} \leq Ce^{-\beta (\eta -2\delta)} \leq Ce^{-\beta \eta};\quad 1 \leq i+2j \leq 4;
\]
and in the initial layer  region $(-\frac{p}{\sqrt{\ve}},\frac{1+q}{\sqrt{\ve}}) \times (0,1]$, simply use
\[
\Bigl \Vert \frac{\partial ^{i+j} \tilde w^*_I}{\partial \eta ^{j}  \partial \zeta  ^i}  \Bigr\Vert _{\tilde N_\delta} \leq C \leq Ce^{-\beta \eta},\qquad \hbox{as} \quad  \eta \leq 1.
\]
 After transforming back to the original variables, we  have thus established the pointwise bounds  (\ref{initial-layer-bounds-der}) on the partial derivatives of $w_I$.

To establish sharper bounds on the space derivatives of $w_I$, we differentiate the differential equation (\ref{initial-layer})  with respect to  the space variable and formulate  parabolic problems for $(w^*_I)_{x}$ and $(w^*_I)_{xx}$.
\begin{eqnarray*}
L^* \bigl[ (w^*_I)_{x}\bigr] = -b^*_{x}w^*_I , \hbox{  in }   Q^*_B ;
  \Vert (w^*_I)_{x}(x,0) \Vert \leq C, \ -p < x < 1+q.
\end{eqnarray*}
The extensions can be constructed so that $\Vert (w^*_I)_{x}(-p,t) \Vert \leq C, \Vert (w^*_I)_{x}(1+q,t)\Vert \leq C$.
Repeating the above argument with the maximum principle and the stretched variables, one can deduce that
$$
 \left\vert \frac{\partial ^{i+j}}{\partial x^{i}\partial t ^{j}} \left (\frac{\partial w^*_I}{\partial x} \right) (x,t) \right\vert  \leq C \left(1+\vr ^{-(i/2 +j)}\right)e^{-\frac{\beta t}{\vr}},\ 0 \leq i+2j \leq 2.
$$
Note further,
\begin{eqnarray*}
&L^* \bigl[ (w^*_I)_{xx}\bigr] = -b^*_{xx}w^*_I -2b^*_{x}(w^*_I)_{x}, \hbox{  in }   Q^*_B \\
& (w^*_I)_{xx}(-p,t)=0, \quad (w^*_I)_{xx}(1+q,t)=0,  \ t \geq 0 \\
&  (w^*_I)_{xx}(x,0)=((y-v^*)^*(x,0))_{xx}, \ -p < x < 1+q.
\end{eqnarray*}
Repeating the argument, one can deduce that
$$
 \left\vert \frac{\partial ^{i+j}}{\partial x^{i}\partial t ^{j}} \left (\frac{\partial^2 w^*_I}{\partial x^2} \right)  (x,t) \right\vert  \leq C \left(1+\vr ^{-(i/2 +j)}\right)e^{-\frac{\beta t}{\vr}},\ 0 \leq i+2j \leq 2,
$$
which complete the proof.
%This will establish the sharper bound on the fourth space derivative  as
%\[
%\vert (w^*_I)_{xxxx} \vert  = \vert (w^*_I)_{txx}-\vr^{-1} (bw^*_I)_{xx} \vert
%\le C (1+\vr^{-1}) {\color{blue}e^{-\frac{\beta t}{\vr}}}.
%\]
 %Finally, the estimates for $(w^*_I)_{xx}$ and $(w^*_I)_{xxxx}$ together with the Mean Value Theorem are used to deduce that $\vert (w^*_I)_{xxx} \vert \le C (1+\vr^{-1/2}){\color{blue}e^{-\frac{\beta t}{\vr}}}.$
\end{proof}

We consider now the initial-boundary  layer component $w_{IB}$ defined in~\eqref{initial-boundary-layera}-\eqref{initial-boundary-layerc}. As $y$ is continuous and the components $v,w_L,w_R, w_I$ are smooth, zero-order  compatibility conditions (for $w_{IB}$) are satisfied. We further decompose the initial-boundary layer term  via
\[
w_{IB} (x,t)= A_1z_1 (x,t)+ w_C(x,t),
\]
where the constant $A_1$ is specified in (\ref{first}) and the function $z_1$ is defined in (\ref{zn}). 
Note that $z_1 \not \in C^{2+\gamma}(\bar Q) $.

\begin{theorem}\label{thm-3}
The initial-boundary layer component $w_{C} \in  C^{2+\gamma}(\bar Q) $. For all $(x,t)\in \bar Q$
\begin{subequations}
\begin{equation}
\vert w_{C}(x,t) \vert \leq Ce^{-\frac{ \beta t}{\vr}}; \end{equation}
and
\begin{eqnarray}
\left \vert \frac{\partial ^{2}w_{C}}{\partial x^{2}}  (x,t) \right\vert  + \left\vert \frac{\partial w_{C}}{\partial t }  (x,t) \right\vert \leq
 C +C\vr ^{-1}\left(e^{-\sqrt{\frac{\beta}{\vr}}x} +  e^{-\sqrt{\frac{\beta}{\vr}}(1-x)}+ e^{-\frac{\beta t}{\ve}}\right)  \\
\left\vert \frac{\partial ^{4}w_{C}}{\partial x^{4}}  (x,t) \right\vert  + \left\vert \frac{\partial ^2w_{C}}{\partial t ^2}  (x,t) \right\vert \leq
 C \vr ^{-1}+C\vr ^{-2}\left(e^{-\sqrt{\frac{\beta}{\vr}}x} +  e^{-\sqrt{\frac{\beta}{\vr}}(1-x)}+ e^{-\frac{\beta t}{\ve}}\right) \label{corner-layer-bounds}
\end{eqnarray}
\end{subequations}
\end{theorem}

\begin{proof}
From  Appendix 2, the initial-boundary layer component $w_{C}$
can be written in the form
\[
w_{C}(x,t)=(A_2 z_2+ A_0\Psi+R_C)(x,t),%b_x(0,0) A_1\Psi _1 ,
\]
where the function $z_2$ is defined in (\ref{zn}) and  the other terms $A_2, \Psi(x,t) $  are defined in  Appendix 2.
From this construction
\begin{eqnarray*}
R_C &=& w_{C} - A_2 z_2- A_0\Psi = w_{IB} - A_1 z_1- A_2 z_2- A_0\Psi \\
&=& (u - A_0 z_0- A_1 z_1- A_2 z_2- A_0\Psi)-(v+w_L+w_R+w_I) \\
&=& y_2 -(v+w_L+w_R+w_I).
\end{eqnarray*}
 In Appendix 2, it is established that $y_2\in  C^{4+\gamma}(\bar Q)$, which implies that $R_C\in  C^{4+\gamma}(\bar Q)$.

The remainder $R_C(x,t)$ satisfies the problem
\begin{eqnarray*}
LR_C&=&Ly_2 -f\\
&=& O(t^2+x^4 +x^2t) A_0z_0 +O(t +x^2) A_2z_2 + O(t +x^2) A_1 z_1 \\
&&+ \ve ^{-1} A_0O(t +x^2)\bigl( O(x^2t+t^2) z_0 + O(xt^2+t^3) (z_0)_x +O(t^3) (z_0)_{xx}\\
&&+ O(t^4) (z_0)_{xxx}\bigr); \\
R_C(x,0) &=& -(w_L^*+w_R^*)(x,0), \quad 0< x< 1, \\ R_C(0,t) &=& - \bigl(w^*_{I} +A_1z_1+ A_2 z_2+ A_0\Psi \bigr)(0,t),  \ t >0, \\
R_C(1,t) &=&- \bigl(w^*_{I} +A_1z_1+ A_2 z_2+ A_0\Psi \bigr)(1,t), \ t >0 .
\end{eqnarray*}
 Although the functions $z_2,  \Psi \not \in C^{4+\gamma}{(\bar Q)}$, we still have  the necessary bounds on the higher derivatives of  these functions. See Appendix 2 for details.  It remains to bound the derivatives of $R_C$.

Using the bounds in (\ref{star1}) we have that
\[
\Bigl \vert \frac{\partial ^{i+2j} }{\partial x^i \partial t^j} (LR_C(x,t)) \Bigr \vert \leq C \ve ^{-j}e^{-\frac{\beta t}{\ve}}, \qquad 0 \leq i+2j \leq 2;
\]
and, from the previous two theorems and the fact that $R_C\in  C^{4+\gamma}(\bar Q)$, we deduce that
\begin{eqnarray*}
\left \vert \frac{\partial ^{j}}{\partial t ^{j}} R_C (0,t) \right \vert
+ \left\vert \frac{\partial ^{j}}{\partial t ^{j}} R_C (1,t) \right\vert
\leq C\ve ^{-j} e^{-\frac{\beta t}{\ve}}; \ 0\leq j \leq 2;
\\
\left\vert \frac{\partial ^{i}}{\partial x ^{i}} R_C (x,0) \right\vert \leq C\ve ^{-i/2} \left(e^{-\sqrt{\frac{\beta}{\vr}}x} +e^{-\sqrt{\frac{\beta}{\vr}}(1-x)}\right); \ 0\leq i \leq 4.
\end{eqnarray*}
 From the maximum principle  we then have that
\[
\vert R_C(x,t) \vert \leq C  e^{-\frac{\beta t}{\ve}}.
\]
Using the stretched variables $ x/\sqrt{\ve}, t/\ve$ and the argument from the proofs of the previous theorems, we can deduce that
\[
\Bigl \vert \frac{\partial ^{i+j} }{\partial x^i \partial t^j} R_C(x,t) \Bigr \vert \leq C \ve ^{-j}e^{-\frac{\beta t}{\ve}}, \qquad 0 \leq i+2j \leq 4.
\]
This completes the proof.
\end{proof}

\section{Numerical Method}

To accurately capture the layers in both space and time, we use a tensor product of two piecewise-uniform Shishkin meshes \cite{fhmos} $\bar Q^{N,M}:= \omega ^N_x \times \omega ^M_t$.
The space Shishkin mesh $\omega ^N_x:=\{x_i\}^N_{i=0}$ is fitted to the two boundary layers by splitting the space domain as follows:
\begin{subequations}\label{mesh_RD}
\begin{equation}
[0, \sigma] \cup [ \sigma ,1-\sigma] \cup [1-\sigma ,1].
\end{equation}
The $N$ space mesh points are distributed in the ratio $N/4:N/2:N/4$ across these three subintervals.
The transition point $ \sigma$ (in space) is taken to be
\begin{equation}
\sigma := \min \left \{ 0.25, 2\sqrt{\frac{\ve}{\beta}} \ln N \right \} .
\end{equation}
The  Shishkin mesh $\omega ^M_t:=\{t_j\}^M_{j=0}$ splits the time domain  into two subintervals $[0,\tau]\cup [\tau, 1]$ and the mesh points in time are distributed equally between these two subintervals.
The transition point $ \tau$ (in time) is taken to be
\begin{equation}
\tau := \min \left \{ 0.5, {\frac{\ve}{\beta}} \ln M \right \}.
\end{equation}
\end{subequations}
We confine our attention to the case where $\sigma < 0.25$ and $\tau < 0.5$.
For the  other case, where $\sigma =0.25$ or $\tau =0.5$,  a classical argument can be applied.  We denote by $Q^{N,M}:=\bar Q^{N,M} \cap Q$ and $\partial Q^{N,M}:= \bar Q^{N,M} \setminus Q^{N,M}$.

We  use a classical  finite difference operator on this mesh to produce the following numerical method:
\begin{subequations} \label{discrete-problem}
\begin{eqnarray}
L^{N,M} Y(x_i,t_j) =\Bigl(f-A_0(b-b(0,0))z_0\Bigr)(x_i,t_j), \qquad (x_i,t_j)\in Q^{N,M}, \\
Y(x_i,t_j)=y(x_i,t_j) , \qquad (x_i,t_j)\in \partial Q^{N,M}, \\
\hbox{where} \quad L^{N,M} Y(x_i,t_j) :=  (\ve D^-_t-\ve\delta^2_x +b(x_i,t_j) I)Y(x_i,t_j).
\end{eqnarray}
\end{subequations}
The finite difference operators are defined by
\begin{eqnarray*}
D^+_x Y(x_i,t _j):= D^-_x Y(x_{i+1},t _j), \quad D^-_x Y(x_i,t _j):=\frac{Y(x_{i},t _j) - Y(x_{i-1},t_{j})}{h_i}, \\
D^-_t Y(x_i,t _j) := \frac{Y(x_i,t _j) - Y(x_{i},t_{j-1})}{k_j }, \quad \delta ^2 _x Y(x_i,t_j):= \frac{(D^+_x - D^-_x)Y(x_{i},t _j)}{\hbar_i},
\end{eqnarray*}
and the mesh steps are $h_i:=x_{i}-x_{i-1}, \, \hbar_i =(h_{i+1}+h_i)/2,\, k_j:= t_j - t_{j-1}$.

We prove below in Theorem~\ref{th:main} that the scheme~\eqref{discrete-problem} is uniformly convergent using a truncation error argument. It is well known that the scheme~\eqref{discrete-problem} satisfies a discrete maximum principle and it is used to derive error estimates from appropriate truncation error estimates.
We recall that the discrete maximum principle establishes that if $Z$ is a grid function that satisfies
$$
L^{N,M} Z\ge 0 \hbox{ on } Q^{N,M} \hbox{ and } Z\ge 0 \hbox{ on } \partial Q^{N,M}, \hbox{ then } Z\ge 0 \hbox{ on } \bar Q^{N,M}.
$$
We now describe how the truncation error estimates are deduced. Away from the transition points, the mesh is uniform and a classical  truncation error argument yields the  bound
\[
\left \vert L^{N,M} (Y-y) (x_i,t_j) \right \vert \leq C h_i^2 \ve \Vert y_{xxxx} \Vert + Ck_j\ve \Vert y_{tt} \Vert, \quad x_i \neq \sigma,1-\sigma.
\]
In addition, the discrete solution  can be  decomposed along the same lines as the continuous solution. That is,
\[
Y= V+W_L+W_R+W_{IB}+ W_I,
\]
where these discrete functions are defined by
$$
L^{N,M} V = Lv(x_i,t_j), \quad
  L^{N,M} W_{L,R,IB,I}(x_i,t_j)=Lw_{L,R,IB,I}(x_i,t_j)
$$
and on the boundary
\[
 V(x_i,t_j)=v(x_i,t_j),\  W_{L,R,IB,I}(x_i,t_j) =w_{L,R,IB,I}(x_i,t_j), \ (x_i,t_j)\in  \partial Q^{N,M}.
\]

\begin{theorem} \label{th:main}
 Let be  $Y$ the solution of the finite difference scheme~\eqref{discrete-problem} and $y$ the solution of the problem~\eqref{eq:ComponentY2}. Then,  the following nodal error estimates are satisfied
\begin{equation} \label{eq:NodalError}
% \Vert y-Y \Vert _{\bar Q^{N,M}} \leq  C\bigl( N^{-1}  + CM^{-1} \ln M \bigr) (\ln M).
 \Vert y-Y \Vert _{\bar Q^{N,M}} \leq  C\left( N^{-2} (\max \{ \ln^2 N , \ln M \} ) + M^{-1} \ln^2 M \right).
\end{equation}
\end{theorem}

\begin{proof}
From the estimates~\eqref{vbounds}, the truncation error for the regular  component is bounded by:
\begin{subequations}\label{TruncationV}
\begin{align}
\bigl \vert L^{N,M} (V-v) (x_i,t_j) \bigl \vert & \leq Ch_i^2 \ve \Vert v_{xxxx} \Vert + Ck_j\ve \Vert v_{tt} \Vert \notag
\\ &\leq  C(N^{-1} \ln N)^2 +CM^{-1}\ln M , \, x_i \neq \sigma,1-\sigma;  \\
\bigl \vert L^{N,M} (V-v) (x_i,t_j) \bigl \vert &\leq C\sqrt{\ve} N^{-1} \ln N +CM^{-1}\ln M , \, x_i = \sigma,1-\sigma .
\end{align}
\end{subequations}

By employing a suitable  barrier function, described in \cite{ria}, we can deduce from the estimates~\eqref{TruncationV} and the discrete maximum principle
\begin{equation}\label{ErrorV}
 \Vert v-V\Vert _{\bar Q^{N,M}}  \le C(N^{-1} \ln N)^2 +CM^{-1}\ln M.
\end{equation}
Using the exponential bounds on the derivatives of the boundary layer components $w_L$ and $w_R$ given in Theorem~\ref{th:vwLRbounds}  and the definition of the space Shishkin mesh $\omega ^N_x$, we bound its truncation errors (see \cite[Theorem 6]{ria} for details) as follows:
\begin{subequations}\label{TruncationWLR}
\begin{align}
%\max_{(x_i,t_j)\in Q^{N,M}}
\left \vert L^{N,M} (W_L-w_L)(x_i,t_j)  \right \vert & \leq C(N^{-1} \ln N)^2 + CM^{-1} \ln M, \\
%\max_{(x_i,t_j)\in Q^{N,M}}
\left \vert L^{N,M} (W_R-w_R)(x_i,t_j)  \right \vert & \leq C(N^{-1} \ln N)^2 + CM^{-1} \ln M,
\end{align}
\end{subequations}
for all $(x_i,t_j)\in Q^{N,M}$. The discrete maximum principle yields the bound
\begin{subequations}\label{ErrorWLR}
\begin{align}
 \Vert  W_L-w_L\Vert _{\bar Q^{N,M}}  & \leq C(N^{-1} \ln N)^2 + CM^{-1} \ln M, \\
 \Vert  W_R-w_R\Vert _{\bar Q^{N,M}}  & \leq C(N^{-1} \ln N)^2 + CM^{-1} \ln M.
\end{align}
\end{subequations}
In the case of the initial layer function, note that
\[
e^{-b(0,0)k_j/\ve} \leq \left(1+\frac{b(0,0)k_j}{\ve}\right)^{-1}
\]
and, using a discrete barrier function $B(t_j)$, we deduce that
\[
\vert W_I(x_i,t_j) \vert \leq \prod _{m=1}^j  \left(1+\frac{\beta k_m}{\ve}\right)^{-1} =: B(t_j),
\]
as
\[
\ve D^-_t B(t_j) +b(x_i,t_j) B(t_j)= (b(x_i,t_j)-\beta) B(t_j) \geq 0.
\]
This barrier function $B(t_j)$ and the estimates~\eqref{sub:initial-layer-bounds} are used to deduce bounds for truncation error associated with the component $w_I$. First,
 outside the initial layer, where $t_j \geq \tau,$ we have
\begin{equation} \label{ErrorWIa}
\vert W_I(x_i,t_j) - w_I(x_i,t_j) \vert \leq \vert W_I(x_i,t_j) \vert + \vert w_I(x_i,t_j) \vert \leq CM^{-1},
\end{equation}
and within the initial layer, where $t_j < \tau$,
\begin{subequations}\label{TruncationWI}
\begin{align}
&\bigl \vert L^{N,M} (W_I-w_I) (x_i,t_j)  \bigl \vert \notag \\
& \hspace{3cm} \leq C(N^{-1} \ln N)^2 + CM^{-1} \ln M, \ x_i \ne \sigma,1-\sigma, \label{TruncationWIb}\\
&\bigl \vert L^{N,M} (W_I-w_I) (x_i,t_j)  \bigl \vert \notag \\
& \hspace{3cm}  \leq C\sqrt{\ve} N^{-1} \ln N + CM^{-1} \ln M, \ x_i=\sigma,1-\sigma. \label{TruncationWIc}
\end{align}
\end{subequations}

As in the case of the continuous initial-boundary layer component $w_{IB}$, we introduce the secondary decomposition  \[
 W_{IB} = W_C+A_1 Z_1,
\] where the components $W_C, Z_1$ are defined as the solutions of
\begin{align*}
 &  L_0^{N,M} W_C:= (\ve D^-_t-\ve \delta^2_x +b(0,0))W_C =0, \ \hbox{on} \ Q^{N,M}  \quad \hbox{and} \quad W_C = w_C \ \hbox{on} \ \partial Q^{N,M}; \\
& L_0^{N,M} Z_1 =0, \ \hbox{on} \ Q^{N,M}  \quad \hbox{and} \quad Z_1 = z_1 \ \hbox{on} \ \partial Q^{N,M}.
\end{align*}
Note that, using a discrete maximum principle,
\[
 \vert W_C(x_i,t_j) \vert \leq C \prod _{m=1}^j  \left(1+\frac{\beta k_m}{\ve}\right)^{-1},  \quad
\vert Z_1(x_i,t_j) \vert \leq C \ve \prod _{m=1}^j  \left(1+\frac{\beta k_m}{\ve}\right)^{-1}.
\]

The error  $W_{IB} -w_{IB}$  is decomposed  into the sum \[
W_{IB} -w_{IB} = W_C -w_C+A_1 (Z_1-z_1),% +b_x(0,0) (K-\chi_{1,0}),
\qquad \ve \vert A_1 \vert \leq C.
\]

From the earlier exponential bounds   on each of the four individual terms $W_C,w_C,Z_1,z_1$ we establish that for  $t_j \geq \tau$
\begin{eqnarray*} \label{ErrorWCa}
\vert W_{IB}(x_i,t_j) - w_{IB}(x_i,t_j) \vert &\leq&  \vert W_{C}(x_i,t_j) \vert + \vert w_{C}(x_i,t_j) \vert  +A_1(\vert Z_1(x_i,t_j) \vert + \vert z_1(x_i,t_j) \vert )\nonumber \\
&\leq&  CM^{-1}.
\end{eqnarray*}

Within the initial layer,  from Theorem~\ref{thm-3}, we have that for $t_j < \tau$,
\begin{subequations}\label{TruncationWC}
\begin{align}
&\bigl \vert L^{N,M} (W_{C}-w_{C}) (x_i,t_j)  \bigl \vert \notag \\
& \hspace{0.2cm}
\leq C(N^{-1} \ln N)^2 + CM^{-1} \ln M, \quad  x_i \notin [\sigma , 1- \sigma ], \\
&\bigl \vert L^{N,M} (W_{C}-w_{C}) (x_i,t_j)  \bigl \vert %\notag \\
%& \hspace{0.2cm}
\leq C\sqrt{\ve}(N^{-1} \ln N) +C\frac{H}{\sqrt{\ve}} e^{-\frac{\beta (\sigma+H)}{\sqrt{\ve}}}+ CM^{-1} \ln M \notag \\
& \hspace{0.2cm} \le C\sqrt{\ve}(N^{-1} \ln N) +C(N^{-1}\ln N) ^2+ CM^{-1} \ln M, \, x_i = \sigma , 1- \sigma, \\
&\bigl \vert L^{N,M}(W_{C}-w_{C}) (x_i,t_j)  \bigl \vert \notag \\
& \hspace{0.2cm} \leq C(N^{-1} \ln N)^2 +C\frac{H^2}{\ve}  e^{-\frac{\beta (\sigma+H)}{\sqrt{\ve}}}+ CM^{-1} \ln M\notag\\
 & \hspace{0.2cm} \le C(N^{-1} \ln N) +C(N^{-1}\ln N) ^2+ CM^{-1} \ln M, \, x_i \in (\sigma , 1- \sigma ),
\end{align}
where $H$ is the mesh width in the coarse region, i.e., $H=2(1-2\sigma)/N=O(N^{-1})$.
\end{subequations}

Collecting the truncation error bounds \eqref{TruncationWI} and~\eqref{TruncationWC} we have for $t_j <\tau$
\begin{align*}
&\bigl \vert L^{N,M} ((W_I+W_{C})-(w_I+w_{C})) (x_i,t_j) \bigl \vert  \\
&\hspace{2cm} \leq  C(N^{-1} \ln N)^2 +CM^{-1}\ln M , \ x_i \neq \sigma,1-\sigma,  \\
&\bigl \vert L^{N,M} ((W_{I}+W_{C})-(w_I+w_{C})) (x_i,t_j) \bigl \vert \\
&\hspace{2cm} \leq C\sqrt{\ve}(N^{-1} \ln N) +CM^{-1}\ln M , \ x_i = \sigma,1-\sigma,
\end{align*}
and
$
\vert (W_I+W_{C})-(w_I+w_{C})) (x_i,\tau) \vert \le C M^{-1}.
$
Hence, the truncation error is first order only along the spatial transition lines $x_i=\sigma,1-\sigma$.
By employing again a suitable  barrier function, described in \cite{ria}, we can deduce for $(x_i,t_j)\in [0,1]\times[0,\tau]$
\begin{equation} \label{ErrorWICb}
\vert (W_I+w_{C})-(w_I+w_{C})) (x_i,\tau) \vert \le  C(N^{-1} \ln N)^2 +CM^{-1}\ln M.
\end{equation}
We employ a different  argument to bound the error \[
 E^j_i :=  A_1 (Z_1-z_1)(x_i,t_j), \quad \hbox{for} \quad  1\leq j \leq M/2. %+b_x(0,0) (K-\chi_{1,0})\bigr)
 \]
From   Appendix 2, $\ve\vert  A_1 \vert \leq C $ and below we bound the truncation
error
\[
 {\cal T}_{i,j}  :=\vert ( L^{N,M}_0 - L_0)  A_1 z_1 (x_i,t_j) \vert
\]
using the bounds on the derivatives of $z_1$ given in Appendix 2.
The truncation error at the first time level $t=t_1$ is thus bounded as follows
\begin{eqnarray*}
{\cal T}_{i,1} \leq C\left  \Vert \frac{\partial ^{2}z_1}{\partial x^{2}}   \right\Vert _{X^1_i}+  C  \left\Vert \frac{\partial  z_1}{\partial t }  \right\Vert _{T^1_i} \leq C
%+  C\ve \left  \Vert \frac{\partial ^{2} \chi _{1,0}}{\partial x^{2}}   \right\Vert _{X^1_i}+  C \ve \left\Vert \frac{\partial  \chi _{1,0}  }{\partial t }  \right\Vert _{T^i_1} \\ &\leq& C
\end{eqnarray*}
where $X^j_i: = [x_{i-1},x_{i+1}]\times \{ t_j \}  , \, T^j_i: =\{ x_i \} \times  [t_{j-1},t_j]$.
For the other time levels $2 \leq j \leq M/2$, we have for $x_i \in (0,\sigma) \cup (1-\sigma ,1)$
\begin{align*}
{\cal T}_{i,j} & \leq C  \left( h^2\left  \Vert \frac{\partial ^{4}z_1}{\partial x^{4}}   \right\Vert _{X^j_i} +   k \left\Vert \frac{\partial ^2 z_1 }{\partial t ^2}  \right\Vert _{T^j_i} \right)
%+  \ve h \left \Vert \frac{\partial ^{3}\chi _{1,0}}{\partial x^{3}}   \right\Vert _{X^j_i}  +   \ve k \left\Vert \frac{\partial ^2 \chi _{1,0} }{\partial  t^2 }  \right\Vert _{T^i_j}  \Bigr) \\
\leq C\left(  h^2 \left( 1 +\frac{\ve }{t_{j}} \right)  +\frac{k}{t_{j-1}} \right)
\end{align*}
and for $x_i \in [\sigma , 1-\sigma ]$
\begin{eqnarray*}
{\cal T}_{i,j} \leq C  \left( \left  \Vert \frac{\partial ^{2}z_1}{\partial x^{2}}   \right\Vert _{X^j_i} +   k \left\Vert \frac{\partial ^2 z_1 }{\partial t ^2}  \right\Vert _{T^j_i} \right)
%+  \ve H \left \Vert \frac{\partial ^{3}\chi _{1,0}}{\partial x^{3}}   \right\Vert _{X^j_i}  +   \ve k \left\Vert \frac{\partial ^2 \chi _{1,0} }{\partial  t^2 }  \right\Vert _{T^i_j}  \Bigr) \\
\leq C \left( z_0(x_{i-1},t_j) + \frac{k}{t_{j-1}}\right).%+ N^{-1}+ \frac{k}{\sqrt{t_{j-1}}}  \bigr).
\end{eqnarray*}
As in  \cite[pg. 916]{Zhemukhov1} and also using $(x- 2t\sqrt{b(0,0)/\ve})^2 \geq 0$, we have
\begin{eqnarray*}
\erfc(z) \leq \frac{e^{-z^2}}{z+\sqrt{z^2 +4/\pi}} \leq Ce^{-z^2}, \ z \geq 0, \quad
e^{-\frac{x^2}{4t}} \leq e^{\frac{b(0,0)t}{\ve}} e^{-\sqrt{\frac{b(0,0)}{\ve}} x};
\end{eqnarray*}
which yield, for all $x_i \geq \sigma$,
\[
z_0(x_{i-1},t_j) \leq Ce^{-\sqrt{\frac{b(0,0)}{\ve}} x_{i-1}}\leq Ce^{\sqrt{\frac{b(0,0)}{\ve}} h}e^{-\sqrt{\frac{b(0,0)}{\ve}} \sigma} \leq CN^{-2}.
\]
Thus, for $x_i \in [\sigma , 1-\sigma ]$, we have the truncation error bound
\[
{\cal T}_{i,j}
\leq C\left(  \frac{k}{t_{j-1}}+  N^{-2}  \right).
\]
We again follow the argument in \cite{Zhemukhov1} and note that at each time level, $t=t_j, \ 1 \leq j \leq M/2$,
\[
-\ve \delta ^2_x  E^j_i + \left(b(x_i,t_j) +\frac{\ve}{k}\right)  E^j_i =  {\cal T}_{i,j} + \frac{\ve}{k}  E^{j-1}_i.
\]
From this we can deduce that
\begin{eqnarray} \label{ErrorZ}
\vert  E^j_i \vert &\leq& C \frac{k}{\ve} \left( {\cal T}_{i,1} +\sum _{n=2}^j  {\cal T}_{i,n} \right) \nonumber\\
&\leq& CM^{-1}  (\ln M ) +C\frac{k}{\ve}M  N^{-2} +  C \left(\frac{k}{\ve} +h^2\right) \int _{s=1}^{M/2} \frac{ds}{s} \nonumber\\
&\leq&   C(M^{-1} \ln  M +N^{-2})\ln  M.
\end{eqnarray}
The error estimates~\eqref{ErrorV}, \eqref{ErrorWLR}, \eqref{ErrorWIa}, \eqref{ErrorWCa},  \eqref{ErrorWICb} and \eqref{ErrorZ} prove the nodal error bound~\eqref{eq:NodalError} and the result follows.
\end{proof}

 One can extend the nodal error estimate~\eqref{eq:NodalError} to a global error estimate by applying the argument in \cite[pp. 56-57]{fhmos}. Note that, in general,  we use a bound on the second space and time derivative (of each component) to establish this global error bound. However, in the case of the terms  $w_{IB}, w_{C}, z_1$, we use the alternative interpolation bound over each rectangle $Q_{i,j} : = (x_i, x_{i+1}) \times (t_j, t_{j+1})$ of the form
\[
\Vert (z - \bar z)(x,t) \Vert _{Q_{i,j} } \leq C(t_{j+1}-t_j) \Vert z_t \Vert _{Q_{i,j}} + C (x_{i+1}-x_i)^2 \Vert z_{xx} \Vert _{Q_{i,j}}.
\]
Moreover, in the case of the term involving $z_1$, note that
\[
A_1 \vert (z_1)
_t (x,t) \vert \leq \frac{C}{\ve} e^{-\beta t/\ve}.
\]

\begin{corollary}
Let  $Y$ be the solution of the finite difference scheme~\eqref{discrete-problem} and $y$ the solution of the problem~\eqref{eq:ComponentY2}. Then,  the following global error estimates are satisfied
 \begin{equation} \label{eq:Global}
\Vert y-\bar Y\Vert \leq  C\left( N^{-2} (\max \{ \ln^2 N , \ln M \} ) + M^{-1} \ln^2 M \right),
\end{equation}
where $\bar Y$ denotes the bilinear interpolant of the discrete function $Y$ from the the values of the grid $\bar Q^{N,M}$ to the domain $\bar Q$.
\end{corollary}
Note that it is the presence  of the fitted mesh (in both space and time) that yields global parameter-uniform convergence.

\section{Numerical results}

Consider the following sample problem from the problem class (\ref{Cproblem}):
\begin{subequations} \label{ex3}
\begin{align}
& b(x,t)= 1+x^2+t, \quad f(x,t)=e^{-x},  \\
& \phi(x)=1-x, \quad  g_L(t)= 0,  \ g_R(t)=-t^2,
\end{align}
\end{subequations}
and the domain is $Q=(0,1)\times(0,1]$. Observe that $1=\phi(0)\ne g_L(0)=0$ and the compatibility conditions  \eqref{First-R} and \eqref{Second-R} are not satisfied at $x=1,t=0$. This problem is a minor variant of a problem considered on the half line $x>0$ in \cite[\S 2]{han} to illustrate the interaction of initial and boundary layers.

The component $y$, which is defined in~\eqref{eq:ComponentY2}, is the solution of the problem
\begin{subequations}  \label{ex3Y}
\begin{align}
&Ly= e^{-x}+(b(x,t)-b(0,0))z_0(x,t) \quad (x,t) \in  Q,\\
&y(0,t)=z_0(0,t), \ y(1,t)=-t^2+z_0(1,t), \ t\ge 0, \\
& y(x,0)=1-x, \ 0<x<1,
\end{align}
\end{subequations}
where we recall that $z_0(x,t)= e^{-\frac{b(0,0) t}{\vr}}\erfc\left(\frac{x}{2\sqrt{t}} \right)$ and $y\in C^0(\partial Q)$. Nevertheless, the component $y$ for this example does not satisfy the first-order compatibility  condition~\eqref{first-x=0} at the corner $(0,0)$.

The exact solution of problem~\eqref{ex3Y} is unknown. In Figure~\ref{fig:U-bl}  the computed approximation $U$ (generated from the finite difference scheme~\eqref{discrete-problem}) to the solution $u$ of Example~\eqref{ex3} is displayed. The solution surface reveals that $u$ has initial and boundary layers. In all the figures of this section we consider the values of $\vr =2^{-12}$ and $N=M=64$.

 \begin{figure}[h!]
\centering
\resizebox{\linewidth}{!}{
	\begin{subfigure}[Entire domain $\bar Q$]{
		\includegraphics[scale=0.5, angle=0]{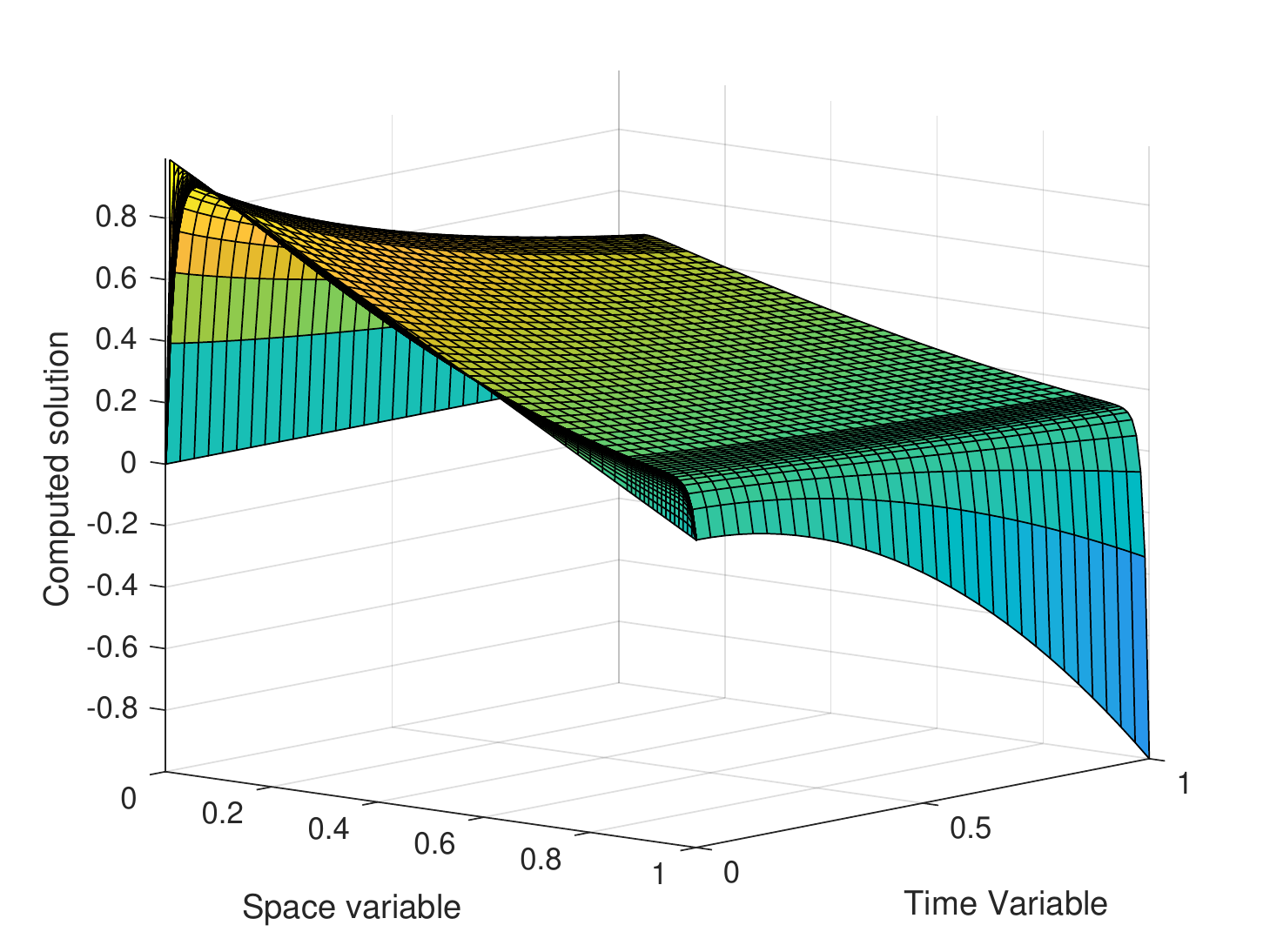}
		}
    \end{subfigure}
\begin{subfigure}[A zoom in on the corner $(0,0)$]{
		\includegraphics[scale=0.5, angle=0]{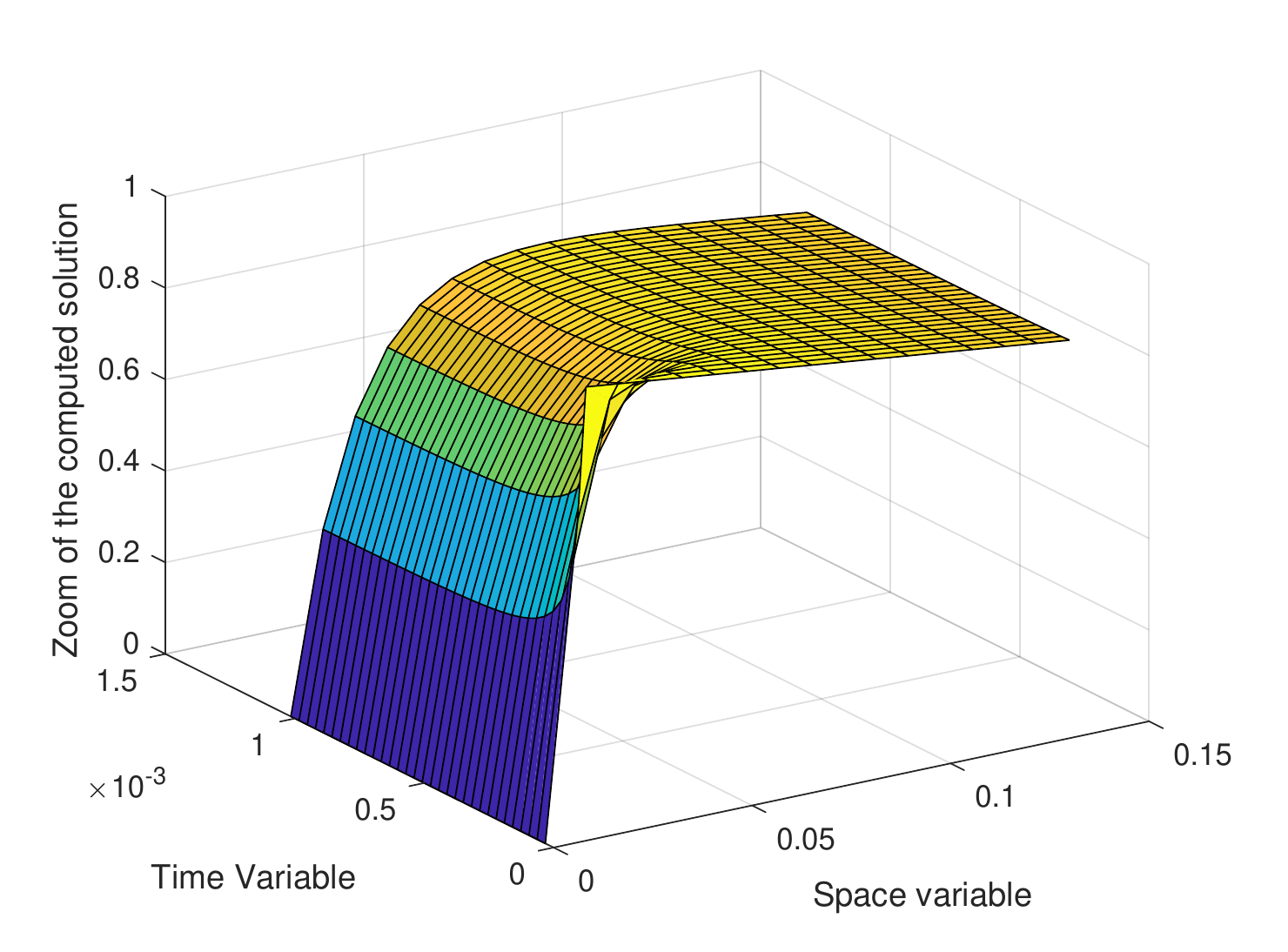}
		}
	\end{subfigure}
}
	\caption{Example~\eqref{ex3}: The numerical approximation to the solution $u$ with $\vr =2^{-12}$ and $N=M=64$}
	\label{fig:U-bl}
 \end{figure}

 The orders of convergence of the finite difference scheme~\eqref{discrete-problem} are estimated using the two-mesh principle~\cite{fhmos}. We denote by  $Y^{N,M}$ and $Y^{2N,2M}$ the computed solutions with~\eqref{discrete-problem} on the Shishkin meshes $Q^{N,M}$ and $Q^{2N,2M}$, respectively. These solutions are used to computed the maximum two-mesh global differences
$$
D^{N,M}_\ve:= \Vert \bar Y^{N,M}-\bar Y^{2N,2M}\Vert _{Q^{N,M} \cup Q^{2N,2M}}
$$
where $\bar Y^{N,M}$ and $\bar Y^{2N,2M}$ denote the bilinear interpolation of the discrete solutions $Y^{N,M}$ and $Y^{2N,2M}$ on the mesh $Q^{N,M} \cup Q^{2N,2M}$. Then, the orders of global  convergence $Q^{N,M}_\ve$ are estimated in a standard way \cite{fhmos}
$$
Q^{N,M}_\ve:=  \log_2\left (\frac{D^{N,M}_\ve}{D^{2N,2M}_\ve} \right).
$$
The uniform  two-mesh global differences $D^{N,M}$ and their corresponding uniform orders of global convergence $Q^{N,M}$ are calculated by
$$
D^{N,M}:= \max_{\ve \in S} D^{N,M}_\ve, \quad Q^{N,M}:=  \log_2\left ( \frac{D^{N,M}}{D^{2N,2M}} \right),
$$
where $S=\{2^0,2^{-1},\ldots,2^{-30}\}$. This is a sufficiently large set of values for the singular perturbation parameter $\vr$ to view the uniform convergence of the scheme~\eqref{discrete-problem}.

The maximum two-mesh global differences $D^{N,M}_\ve$ and the orders of global convergence $Q^{N,M}_\ve$ associated with the problem~\eqref{ex3Y} are displayed in Table~\ref{tb:ComponentYex3Global}.  The uniform two-mesh global differences $D^{N,M}$ and their orders of convergence $Q^{N,M}$ are given in the last row of this table. These numerical results are in line with  the error estimate~\eqref{eq:Global} showing that the method is an almost first-order uniformly global convergent scheme.

\begin{table}[h]
\caption{Example~\eqref{ex3}: Maximum and uniform two-mesh global differences and orders of global convergence associated with the solution $y$ of problem~\eqref{ex3Y}}
\begin{center}{\tiny \label{tb:ComponentYex3Global}
\begin{tabular}{|c||c|c|c|c|c|c|c|}
 \hline & N=64 & N=128 & N=256 & N=512 & N=1024 & N=2048 & N=4096\\
 \hline & M=16 & M=32  & M=64  & M=128 & M=256  & M=512  & M=1024\\
\hline \hline $\vr=2^{0}$
&3.287E-03 &1.822E-03 &9.570E-04 &4.893E-04 &2.467E-04 &1.237E-04 &6.192E-05 \\
&0.851&0.929&0.968&0.988&0.996&0.999&
\\ \hline $\vr=2^{-1}$
&5.854E-03 &3.435E-03 &1.860E-03 &9.647E-04 &4.899E-04 &2.466E-04 &1.236E-04 \\
&0.769&0.885&0.947&0.978&0.991&0.996&
\\ \hline $\vr=2^{-2}$
&9.217E-03 &6.107E-03 &3.510E-03 &1.875E-03 &9.660E-04 &4.896E-04 &2.464E-04 \\
&0.594&0.799&0.905&0.957&0.980&0.991&
\\ \hline $\vr=2^{-3}$
&1.266E-02 &8.951E-03 &6.250E-03 &3.540E-03 &1.877E-03 &9.654E-04 &4.892E-04 \\
&0.500&0.518&0.820&0.915&0.960&0.981&
\\ \hline $\vr=2^{-4}$
&1.266E-02 &9.162E-03 &6.499E-03 &4.197E-03 &2.545E-03 &1.481E-03 &8.388E-04 \\
&0.466&0.495&0.631&0.722&0.781&0.820&
\\ \hline $\vr=2^{-5}$
&2.200E-02 &9.263E-03 &6.521E-03 &4.199E-03 &2.544E-03 &1.480E-03 &8.384E-04 \\
&1.248&0.506&0.635&0.723&0.781&0.820&
\\ \hline $\vr=2^{-6}$
&2.990E-02 &1.127E-02 &6.530E-03 &4.199E-03 &2.543E-03 &1.480E-03 &8.382E-04 \\
&1.408&0.787&0.637&0.723&0.781&0.820&
\\ \hline $\vr=2^{-7}$
&3.123E-02 &1.499E-02 &6.551E-03 &4.204E-03 &2.543E-03 &1.480E-03 &8.382E-04 \\
&1.058&1.194&0.640&0.725&0.781&0.820&
\\ \hline $\vr=2^{-8}$
&3.125E-02 &1.561E-02 &7.510E-03 &4.218E-03 &2.547E-03 &1.480E-03 &8.382E-04 \\
&1.001&1.056&0.832&0.728&0.783&0.821&
\\ \hline $\vr=2^{-9}$
&4.082E-02 &1.562E-02 &7.807E-03 &4.245E-03 &2.550E-03 &1.482E-03 &8.388E-04 \\
&1.386&1.001&0.879&0.735&0.783&0.821&
\\ \hline $\vr=2^{-10}$
&7.276E-02 &2.209E-02 &7.812E-03 &4.299E-03 &2.569E-03 &1.486E-03 &8.397E-04 \\
&1.720&1.499&0.862&0.743&0.790&0.824&
\\ \hline $\vr=2^{-11}$
&7.349E-02 &3.025E-02 &1.121E-02 &4.426E-03 &2.582E-03 &1.490E-03 &8.410E-04 \\
&1.281&1.432&1.341&0.778&0.793&0.825&
\\ \hline $\vr=2^{-12}$
&7.352E-02 &3.025E-02 &1.099E-02 &4.372E-03 &2.611E-03 &1.497E-03 &8.446E-04 \\
&1.281&1.461&1.330&0.743&0.803&0.826&
\\ \hline $\vr=2^{-13}$
&7.354E-02 &3.026E-02 &1.099E-02 &4.372E-03 &2.611E-03 &1.503E-03 &8.428E-04 \\
&1.281&1.461&1.330&0.743&0.797&0.835&
\\
\hline & .&. &.&. &. &. &. \\
 & .&. &.&. &. &. &. \\
 & .&. &.&. &. &. &.
\\ \hline $\vr=2^{-29}$
&7.360E-02 &3.027E-02 &1.100E-02 &4.372E-03 &2.611E-03 &1.503E-03 &8.428E-04 \\
&1.282&1.461&1.331&0.743&0.797&0.835&
\\ \hline $\vr=2^{-30}$
&7.360E-02 &3.027E-02 &1.100E-02 &4.372E-03 &2.611E-03 &1.503E-03 &8.428E-04 \\
&1.282&1.461&1.331&0.743&0.797&0.835&
\\ \hline $D^{N,M}$
&7.360E-02 &3.027E-02 &1.121E-02 &4.426E-03 &2.611E-03 &1.503E-03 &8.446E-04 \\
$Q^{N,M}$ &1.282&1.433&1.341&0.761&0.797&0.832&\\ \hline \hline
\end{tabular}}
\end{center}
\end{table}
Finally, we give some information about the distribution of the errors. In Figure~\ref{fig:YtwoMeshDifferences-bl} we display the  two-mesh nodal differences
$\vert (Y^{N,M}-\bar Y^{2N,2M})(x_i,t_j)\vert$ with $(x_i,t_j)\in Q^{N,M}$; and it is observed that the largest differences occur within the layers and, within the initial layer,  the maximum two-mesh differences occur at the earlier times.
%The numerical solution to problem~\eqref{Cproblem} is displayed in Figure~\ref{fig:U-bl} and it is observed that it exhibits both initial and boundary layers.
 \begin{figure}[h!]
\centering
\resizebox{\linewidth}{!}{
	\begin{subfigure}[Entire domain $\bar Q$]{
		\includegraphics[scale=0.5, angle=0]{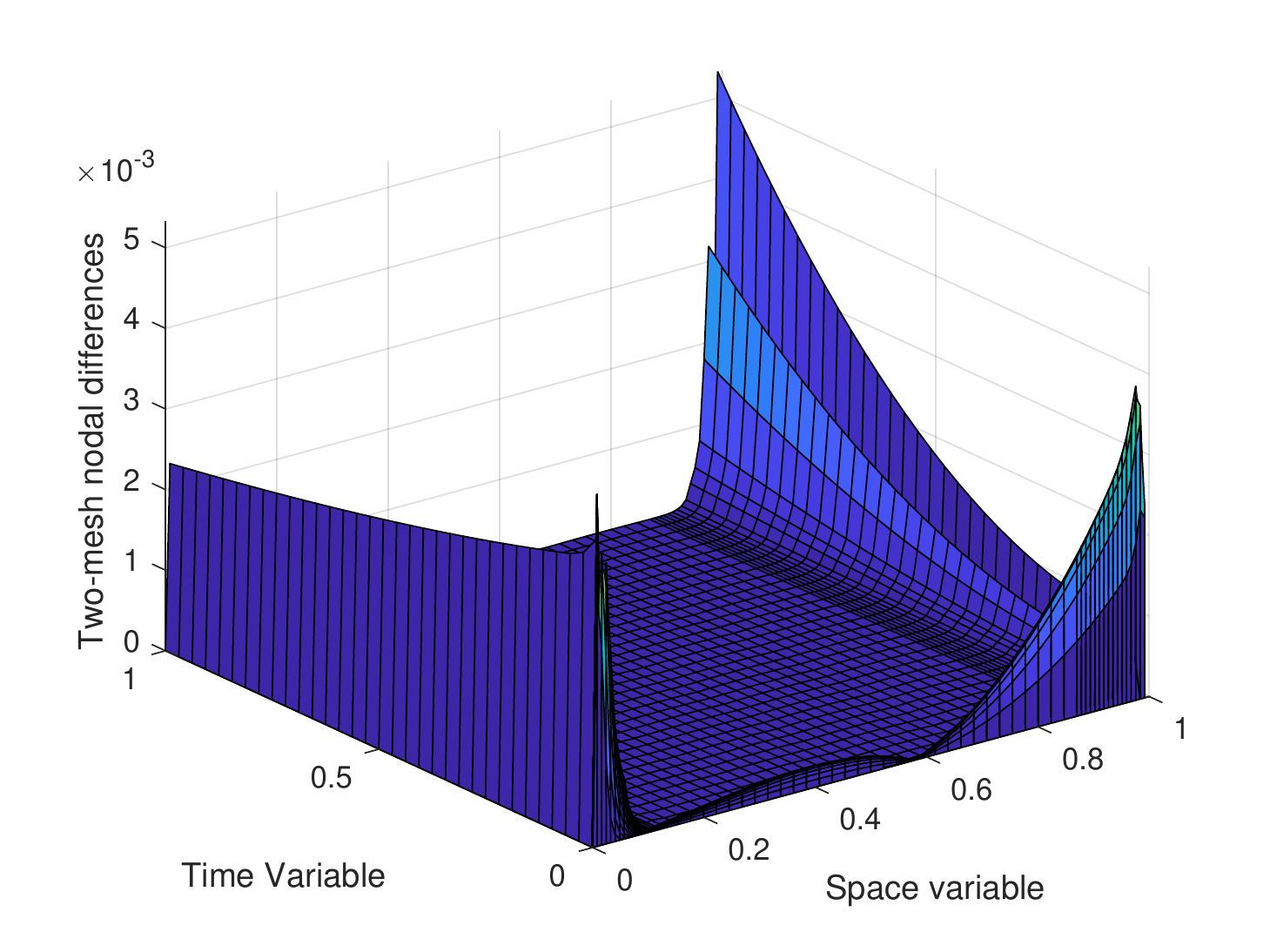}
		}
    \end{subfigure}
\begin{subfigure}[A zoom in on the corner $(0,0)$]{
		\includegraphics[scale=0.5, angle=0]{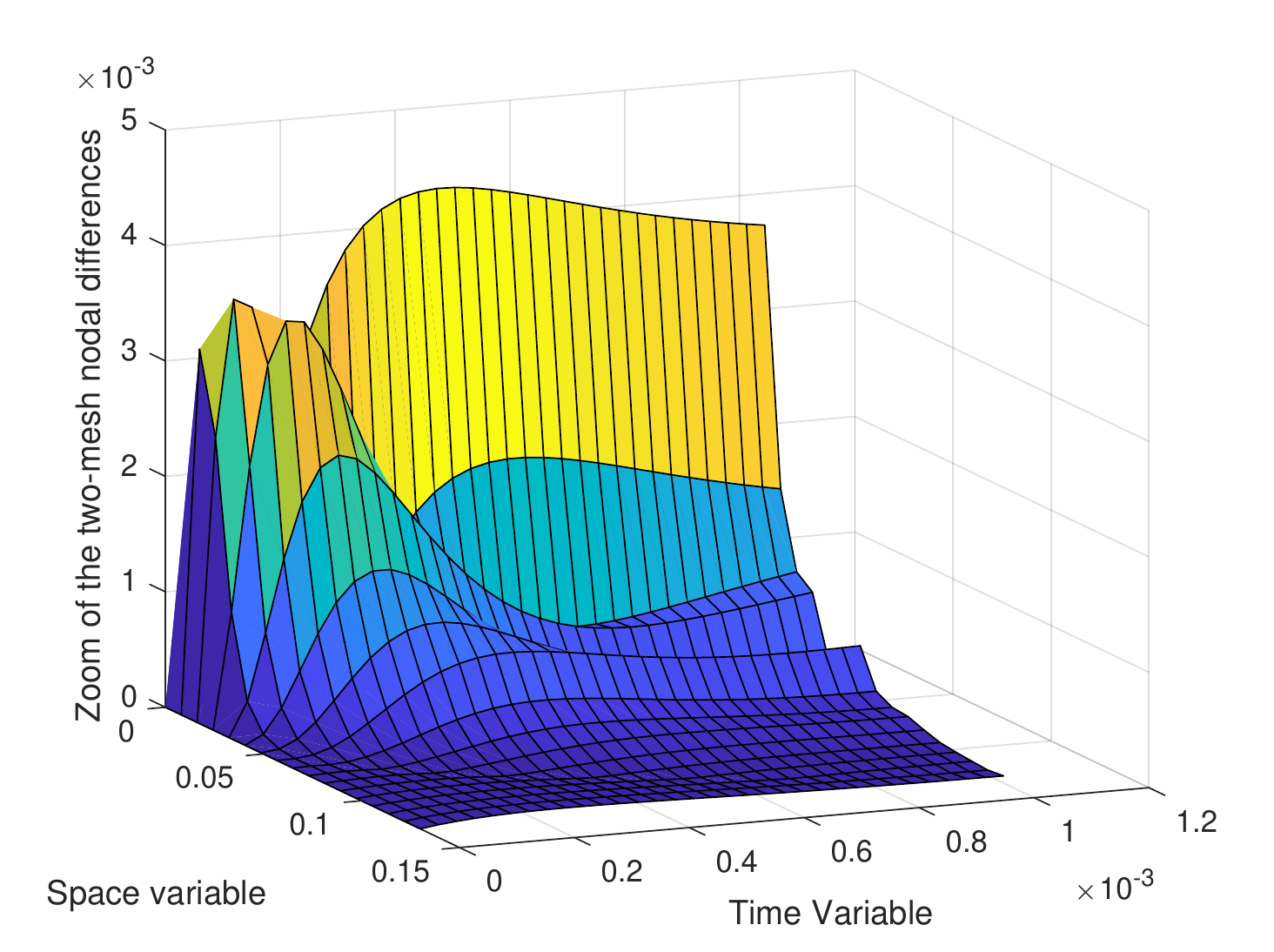}
		}
	\end{subfigure}
}
	\caption{Example~\eqref{ex3}: Two-mesh nodal differences $Y^N-Y^{2N}$ for problem~\eqref{ex3Y} on the coarsest mesh $\bar \Omega^{N,M}$ with $\vr =2^{-12}$ and $N=M=64$}
	\label{fig:YtwoMeshDifferences-bl}
 \end{figure}

\section*{Appendix 1: Compatibility conditions} \label{Appendix:Comp}

In this Appendix, we explicitly write down the compatibility conditions (see e.g. \cite{friedman, ladyz}) of  zero, first and second order associated with
a singularly perturbed parabolic problem in one space dimension. Consider the
following problem: Find $s(x,t)$ such that
\begin{subequations} \label{Classical}
\begin{align}
&Ls=\vr (s_t-s_{xx})+b(x,t)s=g(x,t), \quad (x,t) \in  Q,\\
&s(0,t)=g_L(t), \ s(1,t)=g_R(t) \ t\ge 0, \quad s(x,0)=\phi(x), \ 0<x<1.
\end{align}
\end{subequations}
Below we  place certain regularity and compatibility restrictions on the data in order that the solution   $s \in C^{n+\gamma} (\overline Q),\ n=2,4$.

Level zero-order compatibility conditions corresponds to:
\begin{subequations}\label{compatibility}
\begin{equation}\label{zer0}
\phi(0^+) =g_L(0)  \quad \hbox{and} \quad \phi(1^-) =g_R(0).
\end{equation}
Assuming (\ref{zer0}), we can write
\begin{eqnarray*}
s(x,t)&=& \Phi (x,t) + z(x,t), \quad (x,t) \in  Q \quad \hbox{where} \\
 \Phi (x,t) &:=&  \phi (x)  + (1-x)(g_L(t)-g_L (0)) + x
\bigl(g_R(t)-g_R (0) \bigr);
\end{eqnarray*}
 $ Lz = g - L \Phi;$ and $ z(x,t)=0, \ (x,t) \in \partial Q$. Note that
\begin{eqnarray*}
 L\Phi =  \ve\Bigl((1-x)g_L'(t) +x g'_R(t)  -\phi''(x)\Bigr) +b(x,t) \Phi .
\end{eqnarray*}
From \cite{ladyz}, if  $b,g, L\Phi \in C^{0+\gamma}(\bar Q)$ and the first-order compatibility conditions
\begin{eqnarray}
 \ve (g_R'(0)-\phi''(1^-))+b(1,0)\phi( 1^-)=  g(1,0) \label{first-x=1}
 \\ \ve (g'_L(0)-\phi ''(0^+))+b(0,0)\phi (0^+) = g(0,0) \label{first-x=0}
\end{eqnarray}
are satisfied, then $s \in C^{2+\gamma}(\bar Q)$.
 If  $b,g, L\Phi \in C^{2+\gamma}(\bar Q)$ and  we further assume second-order compatibility conditions
%(so that the mixed derivative $u_{xxt}$ is well defined at $(0,0)$ and $(1,0)$\footnote{\color{magenta}JL: I don't understand very well this sentence. Are the second-order compatibility conditions only needed for this derivative  or are also needed for other higher order derivatives so that  $u \in C^{4+\gamma}(\bar Q)$?}),
such that
\begin{eqnarray}
 (g-L\Phi)_t(0,0^+) + (g-L\Phi)_{xx}(0^+,0)=0; \label{second-at-0}\\
  (g-L\Phi)_t(1,0^+) + (g-L\Phi)_{xx}(1^-,0)=0; \label{second-at-1}
\end{eqnarray}
then the solution of (\ref{Classical}) $s \in C^{4+\gamma}(\bar Q)$. Note that the constraint (\ref{second-at-0}) corresponds to
\begin{align*}
(g_t + g_{xx})(0,0) & =\ve g''_L(0^+) + b(0,0)g'_L(0^+) +b_t(0,0)g_L(0^+)- \ve \phi ^{iv}(0^+) \\
& + 2b_{x}(0,0)\phi'(0^+) + b_{xx}(0,0)\phi(0^+)+b(0,0)\phi''(0^+).
\end{align*}
\end{subequations}

\section*{Appendix 2: Regularity of the function $y_2$}

Consider the solutions $z_n(x,t), n \geq 0$ of the following problems:
\begin{eqnarray*}
 \frac{\partial z_n}{\partial t}-\frac{\partial ^2 z_n}{\partial x^2 } +\frac{b(0,0)}{\ve} z_n= 0,\quad x >0,  t>0 \\
z_n(x,0) =0,\  x \geq 0; \quad z_n(0,t) =t^ne^{-\frac{b(0,0)t}{\ve}},\  t > 0.
\end{eqnarray*}
Note that
\[
z_0(x,t) =\hbox{erfc}\left(\frac{x}{2\sqrt{t}}\right) e^{-\frac{b(0,0)t}{\ve}}; \
 \ve \left\vert \frac{\partial  z_0}{\partial t }  (x,t) \right\vert \leq
 C \frac{1}{t}(\ve +t) e^{-\frac{\beta t}{\ve}};
\]
and
\begin{equation}\label{recurrence}
z_n = n \int _{s=0}^t z_{n-1}(x,s) e^{-\frac{b(0,0)(t-s)}{\ve}}\ ds;\ (z_n)_t +\frac{b(0,0)}{\ve} z_n= n z_{n-1}, n \geq 1.
\end{equation}
Observe that $z_n = v_n e^{-b(0,0) t/\ve}$ and the functions $\{ v_n \} _{n=0}^4$ are explicitly given in \cite[(5), p.538]{Flyer} and we also note that $(v_n)_{xx} = (v_n)_t = n v_{n-1}$.
From a maximum principle, we have  the following bounds:
\begin{equation}
\vert z_n(x,t) \vert\leq t^n e^{-\frac{b(0,0)t}{\ve}} \leq C\ve ^n e^{-\frac{ \beta t}{\vr}} ; \
n=0,1,2.
\end{equation}

Let us list some of the derivatives  of the fundamental function $z_0(x,t)$
\begin{eqnarray*}
\frac{\partial z_0}{\partial x} &=& \frac{-1}{\sqrt{\pi t}} e^{-\frac{x^2}{4t}}e^{-\frac{b(0,0)t}{\ve}} = O\left(\frac{1}{\sqrt{t}}\right) \\
\frac{\partial ^2 z_0}{\partial x^2 } &=&\frac{\partial z_0}{\partial t} +\frac{b(0,0)}{\ve} z_0 =\frac{x}{2t\sqrt{\pi t}}  e^{-\frac{x^2}{4 t}}e^{-\frac{b(0,0)t}{\ve}} = O\left(\frac{1}{ t}\right) \\
\frac{\partial ^3 z_0}{\partial x^3 } &=&
\frac{1}{2t\sqrt{\pi t}} (1-\frac{x^2}{2t}) e^{-\frac{x^2}{4 t}}e^{-\frac{b(0,0)t}{\ve}} = O\left(\frac{1}{ t\sqrt{t}}\right) \\
\frac{\partial ^4 z_0}{\partial x^4 } &=&
 \frac{\partial ^2z_0}{\partial t^2} +2\frac{b(0,0)}{\ve} \frac{\partial z_0}{\partial t}+\left(\frac{b(0,0)}{\ve} \right) ^2 z_0 \\
&=& \frac{-x}{4t^2\sqrt{\pi t}} \left(3 - \frac{x^2}{2t}\right) e^{-\frac{x^2}{4t}}e^{-\frac{b(0,0)t}{\ve}} = O\left(\frac{1}{t^2}\right)\\
\frac{\partial z_0}{\partial t} = O\left(\frac{1}{ \ve}\right) z_0 + O\left(\frac{1}{ t}\right) &,& \frac{\partial ^2z_0}{\partial t^2} = O\left(\frac{1}{ \ve ^2}\right) z_0 + O\left(\frac{1}{ \ve t}\right) + O\left(\frac{1}{ t^2}\right).
\end{eqnarray*}
Observe that the function
\[
h(x,t) := \frac{x}{\sqrt{ t}}  e^{-\frac{x^2}{4t}}, \ t >0, \quad h(x,0) :=0, \  0\leq x \leq 1 ;
\]
is bounded, but not continuous on $\bar Q$.
From this and the explicit expressions for the derivatives of $z_0$ given above, we deduce \footnote{For example,
\[
t ^2\frac{\partial ^2 z_0}{\partial x ^2} \in C^{0+\gamma} (\bar Q)\quad \hbox{and} \quad x^4 \frac{\partial  z_0}{\partial t} \in C^{0+\gamma} (\bar Q).
\]
}
that 
%
%is such that $h\in C^\gamma (\bar Q), \, 0\leq \gamma <1$. {\footnote{
%Let $x=r\cos \theta, t =r\sin \theta$, where $0 < \theta < \pi/2$. Then, if $0 \leq \gamma <1$,
%\[
%\lim _{(x,t) \rightarrow (0,0)} \frac{x}{\sqrt{ t}(\sqrt{x^2+ t})^\gamma }  e^{-\frac{x^2}{4t}} =\lim _{r \rightarrow 0} \frac{\cos \theta }{\sqrt{ \sin \theta }} \frac{r^{(1-\gamma)/2}}{(\sqrt{r \cos ^2 \theta + \sin \theta })^\gamma }  e^{-r\frac{\cos^2\theta }{4\sin \theta }} =0.
%\]
%Hence the function $h$ is H\"{o}lder continuous at the point $(0,0)$. }} 
%From this and the explicit expressions for the derivatives of $z_0$ given above, we deduce that 
 \[
S_{i,j}(x,t):=  x^it^{j}z_0 \in C^{2+\gamma}(\bar Q),\quad  i+2j \geq 4\] and the second derivative in time of these  functions $S_{i,j}$, when $i+2j = 4$, are all bounded on   $\bar Q$, but not continuous. Moreover, for $i+2j \geq 4$, we have
\begin{equation}\label{star1}
\Bigl \vert \frac{\partial ^{n+m} }{\partial x^n \partial t^m} (S_{i,j}(x,t)) \Bigr \vert \leq C \ve ^{-m}e^{-\frac{\beta t}{\ve}}, \qquad 0 \leq n+2m \leq 4;
\end{equation}
and for  all  integers $m,n \geq 1$
\begin{eqnarray*}
L_0(t^nz_0) &=&\ve n t^{n-1} z_0;  \\
L_0(t^nx^mz_0) &=& \ve(x^m n t^{n-1}z_0 - m(m-1)x^{m-2}t^nz_0 - 2m x^{m-1} t^n(z_0)_x).
\end{eqnarray*}
We identify the following set of functions, $\{ \chi _{i,j} \}$,
\[
L_0\chi _{i,j} = S_{i,j}(x,t); \quad i,j=0,1 ,2...
\]
where the first few functions in the set $\{ \chi _{i,j} \}$ are explicitly given as
\begin{subequations}
\begin{eqnarray}
\ve \chi _{1,0}&=& x tz_0+t^2(z_0)_x  \quad \hbox{and} \quad  \vert   \chi _{1,0} (0,t) \vert \leq C \sqrt{t} e^{-\frac{\beta t}{\ve}}; \label{leading-order-sing}\\
\ve\chi _{0,1}&=&
\frac{t^2z_0}{2}  \quad \hbox{and} \quad  \vert  \chi _{0,1} (0,t) \vert \leq Cte^{-\frac{\beta t}{\ve}};   \\
\ve \chi _{2,0}&=&
(x^2t+t^2) z_0+2xt^2(z_0)_x +(4/3)t^3(z_0)_{xx}, \\\hbox{and} \quad
&&   \ve \chi _{2,0} (0,t) =t^2e^{-\frac{b(0,0) t}{\ve}};  \nonumber \\
\ve \chi _{1,1}&=&
\frac{3x t^2z_0+2t^3(z_0)_x}{6} \ \hbox{and} \  \vert  \chi _{1,1} (0,t) \vert \leq  Ct\sqrt{t} e^{-\frac{\beta t}{\ve}} \\
\ve \chi _{3,0}&=&
(x^3 t+3x t^2)z_0+(4t^3+3x^2t^2)(z_0)_x+4xt^3(z_0)_{xx}+2t^4(z_0)_{xxx} \nonumber  \\
&& \vert   \chi _{3,0} (0,t) \vert \leq C t\sqrt{t} e^{-\frac{\beta t}{\ve}} \nonumber .
\end{eqnarray}
\end{subequations}
%We will require the following bounds
%\begin{eqnarray*}
%\ve \left \vert \frac{\partial ^{3}\chi _{1,0}}{\partial x^{3}}  (x,t) \right\vert   &\leq&
 %C e^{-\frac{\beta t}{\ve}}; \quad  \ve \left \vert \frac{\partial ^{4}\chi _{1,0}}{\partial x^{4}}  (x,t) \right\vert  \leq C \frac{1}{\sqrt{t}} e^{-\frac{\beta t}{2\ve}}\\
 %\ve \left\vert \frac{\partial  \chi _{1,0} }{\partial t }  (x,t) \right\vert &\leq&
 %C \sqrt{t}e^{-\frac{\beta t}{2\ve}},  \quad \hbox{using}\quad r\erfc(r) \leq C, \ \forall r \geq 0;
%\\
 %\ve \left\vert \frac{\partial ^2 \chi _{1,0} }{\partial t ^2}  (x,t) \right\vert &\leq&
 %C \frac{1}{\sqrt{t}} e^{-\frac{\beta t}{2\ve}}, \quad \hbox{using}\quad  x\sqrt{t}e^{-\frac{x^2}{t}}e^{-\frac{\beta t}{\ve}}\leq C\ve e^{-\frac{\beta t}{2\ve}}  .
%\end{eqnarray*}
Using the recurrence relation (\ref{recurrence}) and the above properties of $z_0$ we have that
\begin{subequations}
\begin{eqnarray}
 \left \vert \frac{\partial ^{2}z_1}{\partial x^{2}}  (x,t) \right\vert  +  \left\vert \frac{\partial  z_1}{\partial t }  (x,t) \right\vert &\leq&
 C e^{-\frac{\beta t}{\ve}}; \\
 \left \vert \frac{\partial ^{3}z_1}{\partial x^{3}}  (x,t) \right\vert  &\leq&
 C \frac{1}{\sqrt{t}}  e^{-\frac{x^2}{4t}} e^{-\frac{b(0,0) t}{\ve}};  \\
 \left \vert \frac{\partial ^{4}z_1}{\partial x^{4}}  (x,t) \right\vert  + \ve \left\vert \frac{\partial ^2 z_1}{\partial t ^2}  (x,t) \right\vert &\leq&
 C \frac{1}{t} (\ve +t)e^{-\frac{\beta t}{\ve}};  \\
 \left \vert \frac{\partial ^{4}z_2}{\partial x^{4}}  (x,t) \right\vert  +  \left\vert \frac{\partial  ^2z_2}{\partial t ^2}  (x,t) \right\vert &\leq&
 C e^{-\frac{\beta t}{\ve}} .
\end{eqnarray}
\end{subequations}
Note also that
\begin{align*}
&  \left \vert \frac{\partial ^2 z_0}{\partial t^2} (x,t) \right \vert \leq \frac{C}{(\ve +t)^2}e^{-\frac{b(0,0)t}{\ve}}; \ \left \vert \frac{\partial ^2 z_1}{\partial t^2}(x,t) \right \vert \leq \frac{C}{t}e^{-\frac{b(0,0)t}{\ve}}; \\
&  \left \vert \frac{\partial ^2 z_2}{\partial t^2}(x,t) \right \vert \leq Ce^{-\frac{b(0,0)t}{\ve}}.
\end{align*}

Recall that the solution   $u$  of problem (\ref{Cproblem}) is discontinuous at the point $(0,0)$ and by subtracting the discontinuous function  $A_0z_0$, we see that $y=u-A_0z_0$ satisfies zero order compatibility at the point $(0,0)$.
Hence the solution $y$ of problem (\ref{eq:ComponentY2}) is a continuous function. By subtracted off appropriate multiples $A_n$ of $z_n$ (see (\ref{recurrence})) from $u$ we can satisfy up to the $n^{th}$ order compatibility conditions at the point $(0,0)$. Since $L_0z_n=0$, one can check that
\[
L(u-A_0z_0), L(u-A_0z_0- A_1z_1)  \in C^{0+\gamma}(\bar Q)
\]
and this implies that  $(u-A_0z_0)\in C^{2+\gamma}(\bar Q)$.
 Thus $y \in C^{0}(\bar Q)$ and $(y-A_1z_1)\in C^{2+\gamma}(\bar Q)$, but $(y- A_1z_1-A_2z_2)  \not \in C^{4+\gamma}(\bar Q)$.
We  now define the following two functions
\[
y_1:=y-A_1z_1 \quad \hbox{and} \quad y_2:= y_1- A_2 z_2%-b_x(0,0) (A_0\chi _{1,0} +A_1\Psi _1 )
- A_0\Psi ;
\]
where
\begin{eqnarray*}
\Psi (x,t) := b_t(0,0) \chi _{0,1} + \frac{b_{xx}(0,0)}{2} \chi _{2,0} +    b_{xt}(0,0) \chi _{1,1} +  \frac{b_{xxx}(0,0)}{6} \chi _{3,0} \\
\hbox{and} \quad L_0\Psi = \left(b_t(0,0) t + \frac{b_{xx}(0,0)}{2} x^2 +   b_{xt}(0,0) xt +  \frac{b_{xxx}(0,0)}{6} x^3\right)z_0.
%\\ \ve \Psi _1 (x,t) :=   (x tz_1+t^2(z_1)_x), \quad L_0 \Psi _1 (x,t) =  x  z_1
\end{eqnarray*}
%where $Z$ and $S$ satisfy, respectively, the problems
%\begin{eqnarray*} L_0Z=0, x, t >0; \qquad Z(0,t) = t e^{-\frac{b(0) t}{\vr}}, \ t>0;\quad  Z(x,0)=0,\ x >0 \\
%L_0Z_1=0, x, t >0; \qquad Z_1(0,t) = t ^2e^{-\frac{b(0) t}{\vr}}, \ t>0;\quad  Z_1(x,0)=0,\ x >0 \\
%L_0S=ts(x,t), x, t >0; \qquad S(0,t) = 0, t>0,\quad  S(x,0)=0,\ x >0 %\\ S_R(x,t)=S(1-x,t).
%\end{eqnarray*}
We shall see below that the additional term $A_0\Psi$ %$b_x(0,0) (A_0\chi _{1,0} +A_1\Psi _1 )+ A_0\Psi$
has been included so  that  $y_2\in C^{4+\gamma}(\bar Q) $. The amplitude $A_0$ has been determined above. Below we specify the amplitudes $A_1$ and $A_2$.
Observe that the function $y_1$ satisfies
\begin{eqnarray*}
Ly_1 = f - (b(x,t)-b(0,0))\bigl(A_1z_1 + A_0z_0 \bigr); \\
y_1(0,t) =g_L(t)  -(A_0z_0 +A_1z_1)(0,t);\  y_1(1,t)= g_R(t) - (A_0z_0 +A_1z_1)(1,t); \\
y_1(x,0)= \phi (x).
\end{eqnarray*}
 From~\eqref{first-x=0} in the first Appendix, first order compatibility is satisfied (for $y_1$) if $A_1$ is such that
\begin{equation}\label{first}
f(0,0) = \ve (g_L'(0^+) - A_1 - \phi _{xx}(0^+)) + b(0,0)(A_0+ \phi (0)).
\end{equation}
In general, $A_1 = O(\ve ^{-1})$.
Since $Ly_1 \in   C^{0+\gamma}(\bar Q) $, then $y_1 \in  C^{2+\gamma}(\bar Q) $.

Next we move onto the regularity of $y_2$. Note first that, since $b_x(0,0)=0$,
\[
(L-L_0)z_n = \left(b_t(0,0) t + \frac{b_{xx}(0,0)}{2} x^2 +   b_{xt}(0,0) xt +  \frac{b_{xxx}(0,0)}{6} x^3\right)z_n + \ \hbox{H.O.T.} .
\]
Hence,
\begin{eqnarray*}
Ly_2 &=& Ly_1- (L-L_0)(A_2z_2 + A_0\Psi ) - A_0 L_0\Psi
\\
&=& f- (L-L_0)\bigl(A_1z_1 + A_0z_0 \bigr)-A_0 L_0\Psi - (b(x,t)-b(0,0))(A_2z_2 + A_0\Psi )
\\
&=& f  +O(t^2+x^4 +x^2t) A_0z_0- (L-L_0)\bigl(A_1z_1  \bigr)- (b(x,t)-b(0,0))(A_2z_2 + A_0\Psi)
\\ &=& f +O(t^2+x^4 +x^2t) A_0z_0 +O(t +x^2) A_2z_2 +  O(t +x^2) A_1 z_1 ;\\
&&+ \ve ^{-1} A_0O(t +x^2)\bigl( O(x^2t+t^2) z_0 + O(xt^2+t^3) (z_0)_x +O(t^3) (z_0)_{xx}+ O(t^4) (z_0)_{xxx}\bigr)
\end{eqnarray*}
and the boundary and initial conditions are
\begin{eqnarray*}
&&y_2(0,t)=g_L(t) - \bigl(A_0z_0+ A_1z_1+ A_2 z_2+ A_0\Psi \bigr)(0,t)  \\
&&=g_L(t)  - \left(A_0+A_1t+ A_2 t^2\right)e^{-\frac{b(0,0)t}{\ve}} - A_0\frac{t^2}{2\ve}  b_t(0,0) e^{-\frac{b(0,0)t}{\ve}} \\
&&- A_0\frac{t^2}{2\ve} \left(b_{xx}(0,0) \left(z_0+\frac{4t}{3}\frac{\partial ^2 z_0}{\partial x^2}\right)+    \frac{2b_{xt}(0,0) t}{3} \frac{\partial  z_0}{\partial x}+  \frac{2b_{xxx}(0,0)}{3} \left(2t\frac{\partial  z_0}{\partial x} +t^2\frac{\partial ^3 z_0}{\partial x^3}\right)\right)(0,t);\\
&&y_2(1,t)= y(1, t)-\bigl(A_1z_1+ A_2 z_2+ A_0\Psi \bigr)(1,t);\quad
y_2(x,0)=\phi (x).
\end{eqnarray*}
Note that $\lim _{t \rightarrow 0^+} y_2(0,t) = g_L(0^+) -A_0$ and
\begin{align*}
\lim _{t \rightarrow 0^+} \frac{\partial y_2(0,t)}{\partial t}
&=g'_L(0^+)  - A_1 +A_0\frac{b(0.0)}{\ve}  \\
\lim _{t \rightarrow 0^+} \frac{\partial ^2 y_2(0,t)}{\partial t^2}
&=g''_L(0^+) +2\ve ^{-1}A_1 b(0,0)- 2A_2  -2\vr^{-1} A_0   (b_t(0,0) +b_{xx}(0,0)) \\
& -A_0(\frac{b(0,0)}{\ve })^2.
\end{align*}
First order compatibility is satisfied (for $y_2$) if $A_1$ is  such that~\eqref{first} is satisfied
and the above construction has been designed in order  that $Ly_2 \in   C^{2+\gamma}(\bar Q) $.
Finally second order compatibility is satisfied (for $y_2$) if $A_2$ is such that
\begin{align}
&\vr (g_L''(0^+) - \phi ^{iv}(0^+))
+(A_1+g'_L(0^+)) b(0,0)- 2\ve A_2  -2 A_0   (b_t(0,0) +b_{xx}(0,0))  \nonumber
\\
& - A_0\ve ^{-1} b^2(0,0) + b_t(0,0)(g_L(0)-A_0)  + b_{xx}(0,0)\phi(0^-)+b(0,0) \phi''(0^+)  \nonumber \\
& \quad =\bigl(f_t+f_{xx} \bigr)(0,0). \label{A2}
\end{align}
Note that
$
A_0= O(1), \quad A_1= O(\ve ^{-1}) \quad \hbox{and} \quad A_2= O(\ve ^{-2}).
$
By this construction we have that $y_2 \in   C^{4+\gamma}(\bar Q)$. Moreover,
\begin{equation}\label{star2}
\Bigl \vert \frac{\partial ^{i+j} }{\partial x^i \partial t^j} (Ly_2(x,t)) \Bigr \vert \leq C \ve ^{-j}e^{-\frac{\beta t}{\ve}}, \quad 0 \leq i+2j \leq 2.
\end{equation}


\begin{thebibliography}{99}

%\bibitem{bobisud-a} L. Bobisud, Second--order linear parabolic equations with a small parameter, {\em Arch. Rational Mech. Anal.}, {\bf 27}, 1967, 385--397.

%\bibitem{bobisud} L. Bobisud, Parabolic equations with a small parameter and discontinuous data, {\em J. Math. Anal. Appl.}, {\bf 26}, 1969, 208--220.



\bibitem{temam} Q. Chen, Z. Qin and R. Temam, Treatment of incompatible initial and boundary data for parabolic equations
in higher dimensions,  {\em Math. Comp.}, {\bf 80}, 276, 2011, 2071--2096.



\bibitem{fhmos}  P.A. Farrell, A.F. Hegarty, J.J.H. Miller, E. O'Riordan and  G.I. Shishkin, {\em Robust computational techniques for boundary layers},
CRC Press, 2000.

\bibitem{han2} M. Fei,  D. Han  and  X. Wang, Initial-boundary layer associated with the nonlinear Darcy-Brinkman-Oberbeck-Boussinesq system,
{\em Phys. D}, {\bf 338}, 2017, 42--56.

\bibitem{Flyer} N. Flyer and B. Fornberg, Accurate numerical resolution of transients in initial-boundary
value problems for the heat equation,  {\em J. Comp. Physics} {\bf 184}, 2003, 526--539.


\bibitem{bulg2012} J.L. Gracia and  E. O'Riordan, A singularly perturbed reaction-diffusion problem with incompatible boundary-initial data,  Lecture Notes in Computer Science, I. Dimov, I. Farago, and L. Vulkov (Eds.): Numerical Analysis and Its Applications: 5th Int. Conf., NAA 2012, Lozenetz, Bulgaria,  v. 8236, 303--310. Springer (2013).

\bibitem{amc} J.L. Gracia and  E. O'Riordan, A singularly perturbed parabolic problem with a layer in the initial condition, {\em Appl. Math. Comp.} {\bf 219}, 2012, 498--510.


%\bibitem{paul} P.A.\ Farrell, P.W. Hemker and G.I. Shishkin
%Discrete approximations for singularly perturbed boundary problems with parabolic layers III, {\em J.  Comput. Math.}, {\bf 14} (3),  1996, 273--290.

\bibitem{friedman} A. Friedman, {\em Partial differential equations of parabolic type}, Prentice-Hall, Englewood Cliffs, N.J., (1964).

\bibitem{han} D. Han  and  X. Wang, Initial-boundary layer associated with the nonlinear
              {D}arcy-{B}rinkman system,  {\em J. Differential Equations},  {\bf 256}, (2), 2014, {609--639}.

\bibitem{hemker} P.W. Hemker and G.I. Shishkin, Approximation of parabolic PDEs with a discontinuous initial condition, {\em East-West J. Numer. Math}, {\bf 1}, 1993, 287--302.

\bibitem{hemker2}  P.W. Hemker and G.I. Shishkin,
Discrete approximation of singularly perturbed parabolic PDEs with a discontinuous initial condition, {\em Comp. Fluid Dynamics J}, {\bf 2}, 1994,  375--392.

\bibitem{review-GIS} N.V. Kopteva and E. O'Riordan, Shishkin meshes in the numerical solution of singularly perturbed differential equations, {\em Int. J. Numer. Anal. Model.},  {\bf 7}, 2010, 393--415.

\bibitem{ladyz} O.A. Ladyzhenskaya, V.A. Solonnikov  and  N.N. Ural'tseva,  {\em Linear and
quasilinear equations of parabolic type}, Transactions of Mathematical Monographs, {\bf 23}, American Mathematical Society, (1968).



\bibitem{cwi987b}  J.J.H. Miller  and  E. O'Riordan, The necessity of fitted operators and Shishkin meshes for resolving thin layer phenomena,
{\em CWI Quarterly}, {\bf 10}, 207--213,  1997.

\bibitem{ria} J.J.H. Miller, E. O'Riordan, G.I. Shishkin and L.P. Shishkina, \newblock{Fitted mesh methods for problems with parabolic boundary layers}, {\em Mathematical Proceedings of the Royal Irish Academy}, {\bf 98A}, 1998, 173--190.

\bibitem{mos2}  J.J.H. Miller, E. O'Riordan and G.I. Shishkin, {\em Fitted Numerical  Methods for Singular Perturbation Problems}, World-Scientific,
Singapore (Revised edition),  2012.

\bibitem{hemkerb} E. O'Riordan and G. I. Shishkin, Parameter uniform numerical methods for singularly perturbed elliptic problems with parabolic boundary layers, {\em Applied Numerical Mathematics}, {\bf 58}, 2008, 1761-1772.

%\bibitem{rst2} H.-G.~Roos, M.~Stynes  and L.~Tobiska, {\em Robust numerical methods for singularly perturbed differential equations},
%Springer-Verlag, Berlin Heidelberg, 2008.

\bibitem{styor}  M.\ Stynes and  E.\ O`Riordan,
\newblock {Uniformly convergent difference schemes for singularly perturbed
parabolic diffusion-convection problems without turning points},
\newblock {\em Numer. Math.}, {\bf 55},   1989, 521--544.

  \bibitem{Zhemukhov1} U.Kh. Zhemukhov,  Parameter-uniform error estimate for the implicit four-point scheme for a singularly perturbed heat equation with corner singularities. Translation of Differ. Uravn. 50 (2014), {\bf 7}, 923–-936. Differ. Equ. 50 (2014), {\bf 7}, 913--926.

 \bibitem{Zhemukhov2} U.Kh. Zhemukhov,  On the convergence of the numerical solution of an initial-boundary value problem for the heat equation in the presence of a corner singularity in the derivatives of the solution. (in Russian) Vestnik Moskov. Univ. Ser. XV Vychisl. Mat. Kibernet. 2013, {\bf 4}, 9--18, 50; translation in Moscow Univ. Comput. Math. Cybernet. 37 (2013),{\bf 4}, 162--171.

\end{thebibliography}
\end{document}